\documentclass[11pt]{article}
\usepackage{amscd}
\usepackage{amsfonts}
\usepackage{amsmath}
\usepackage{amssymb}
\usepackage{amsthm}
\usepackage{bbm}
\usepackage{CJK}
\usepackage{fancyhdr}
\usepackage{graphicx}
\usepackage{indentfirst}
\usepackage{latexsym,bm}
\usepackage{mathrsfs}
\usepackage{xypic}
\usepackage[square, comma, sort&compress, numbers]{natbib}
\usepackage[colorlinks,linkcolor=red, anchorcolor=blue]{hyperref}
\allowdisplaybreaks[4]
\newtheorem{theorem}{Theorem}[section]
\newtheorem{lemma}[theorem]{Lemma}
\newtheorem{definition}[theorem]{Definition}
\newtheorem{proposition}[theorem]{Proposition}
\newtheorem{example}[theorem]{Example}

\newtheorem{remark}[theorem]{Remark}

\usepackage[top=1in,bottom=1in,left=1.25in,right=1.25in]{geometry}
\def\<{\langle}
\def\>{\rangle}
\def\a{\alpha}
\def\B{\Bbbk}
\def\b{\beta}

\def\c{\cdot}

\def\D{\Delta}

\def\lr{\longrightarrow}

\def\m{\mapsto}
\def\n{\natural}
\def\o{\otimes}
\def\om{\omega}

\def\r{\rho}
\def\ra{\rightarrow}

\def\s{\sum}

\def\tr{\triangleright}
\def\tl{\triangleleft}

\def\v{\varepsilon}

\date{}
\begin{document}
\renewcommand{\baselinestretch}{1.2}
\renewcommand{\arraystretch}{1.0}
\title{\bf BiHom-L-R-smash biproduct and BiHom-Yetter-Drinfel'd-Long category}
 \date{}
\author {{\bf Dongdong Yan\footnote {Corresponding author:  Ydd150365@163.com}  }\\
{\small School of Mathematics and Physics, Nanjing Institute of Technology, Nanjing, }\\
{\small  Jiangsu 211167, P. R. of China}\\
} \maketitle
\begin{center}
\begin{minipage}{12.cm}

\noindent{\bf Abstract.} In this article, we first introduce the notion of BiHom-L-R-$\binom{m,n,p,q}{s,t,u,v}$-smash biproduct over a BiHom-Hopf algebra, denoted by $D\natural H$, where $m,n,p,q,s,t,u,v\in \mathbb{Z}$, and give the sufficient condition for $D\natural H$ to be a BiHom-bialgebra. Furthermore, we describe the concept of BiHom-$\binom{m,n,p,q}{s,t,u,v}$-Yetter-Drinfel'd-Long bimodule via BiHom-L-R-$\binom{m,n,p,q}{s,t,u,v}$-smash biproduct bialgebra, and prove that the category $\mathcal{LR}(H)(m,n,p,q)$ of BiHom-$\binom{m,n,p,q}{s,t,u,v}$-Yetter-Drinfel'd-Long bimodule is a strict braided monoidal category. Finally, for a finite-dimensional BiHom-Hopf algebra H, $\mathcal{LR}(H)\binom{m,n,p,q}{s,t,u,v}$ is isomorphic to the BiHom-$\binom{s,t}{p,q}$-Yetter-Drinfel'd category ${}_{H\otimes H^*}^{H\otimes H^*}\mathcal{YD}\binom{s,t}{p,q}$ as braided monoidal categories.
\\

\\

 \end{minipage}
 \end{center}
 \normalsize\vskip1cm

\section{Introduction}
\def\theequation{1.\arabic{equation}}
\setcounter{equation} {0}
It is well known that the Radford biproduct admits a categorical interpretation originating from Majid: a pair \((H,A)\) is admissible (see \cite{R85}) if and only if A is a bialgebra in the Yetter–Drinfel’d category \({}_H^H\mathcal{YD}\). For L-R-admissible pairs, Panaite and Van Oystaeyen \cite{PV10} established an analogous categorical picture and introduced a prebraided monoidal category \(\mathcal{LR}(H)\) (which becomes braided whenever H possesses a bijective antipode). They further proved that \((H,D)\) forms an L-R-admissible pair satisfying the additional compatibility condition
\[b_{(0)}\triangleleft b'_{[-1]}\otimes b_{(1)}\triangleright b'_{[0]}=b\otimes b',\quad \forall\,b,b'\in D,\]
precisely when D is a bialgebra within \(\mathcal{LR}(H)\). Here the L-R-admissible condition guarantees that \(D\otimes H\) carries a bialgebra structure via the L-R-smash product and L-R-smash coproduct; this bialgebra is written as \(D\n H\) and referred to as the L-R-smash biproduct, which contains the classical Radford biproduct as a special instance. In the setting where H is a finite-dimensional bialgebra, Lu and Wang \cite{LW17} showed that \(\mathcal{LR}(H)\) and \({}_{H\otimes H^*}^{H\otimes H^*}\mathcal{YD}\) are isomorphic monoidal categories (and these categories are braided provided that H is a Hopf algebra). Subsequently, Lu and Zhang \cite{LZ18} extended the above results to the framework of Hom-Hopf algebras.

Hom-algebras were first introduced by Makhlouf and Silvestrov in \cite{MS08}. In this framework, associativity is replaced by Hom-associativity: \(\alpha(a)bc = (ab)\alpha(c)\). Hom-coassociativity for Hom-coalgebras may be defined analogously (see \cite{MS10}). The definitions of Hom-bialgebras and Hom-Hopf algebras have also been proposed and systematically developed, see \cite{FK20, G10, LM14, LS14,  MP14, MP15}. From the perspective of monoidal categories, Caenepeel and Goyvaerts investigated Hom-structures in \cite{CG11} and introduced monoidal Hom-algebras, monoidal Hom-coalgebras and related objects within symmetric monoidal categories. These notions differ slightly from the aforementioned Hom-algebras and Hom-coalgebras. In \cite{MLC20}, Ma, Liu and Chen defined the \((m+n)\)-Yetter–Drinfel’d category \({}_H^H\mathcal{YD}(m,n)\) with \(m,n\in\mathbb{Z}\) via Radford \((m,n)\)-biproduct monoidal Hom-Hopf algebras, and proved that \({}_H^H\mathcal{YD}(m,n)\) is a braided monoidal category. In \cite{Y26}, Yan described L-R-$(m, n, p, q)$-smash biproduct and
$(m, n, p, q)$-Yetter–Drinfel’d–Long category with \(m,n,p,q\in\mathbb{Z}\). For further relevant work, we refer the reader to \cite{GW24,MLC18,MLC20}.

BiHom-(co)algebras and BiHom-bialgebras were investigated by Graziani et al. in \cite{GMMP15}. This framework provides a more general definition than that of Hom-bialgebras. More precisely, a BiHom-bialgebra is a generalized bialgebra for which the associativity and unit conditions are twisted by two automorphisms $\alpha$ and $\beta$, while the coassociativity and counit conditions are twisted by another two automorphisms $\varphi$ and $\psi$. It reduces to an ordinary Hom-bialgebra when $\alpha = \beta = \varphi = \psi$, and to a monoidal Hom-bialgebra when $\alpha^{-1} = \beta^{-1} = \varphi = \psi$. In \cite{ZWC24}, Zhang et al. investigated BiHom-type modules, Yetter-Drinfel'd modules and Drinfeld doubles with parameters. Further work on BiHom-type algebras can be found in \cite{FL18,GZW18,LMMP20,ZW20} among others. 

BiHom-Hopf algebras generalize Hom-Hopf algebras by introducing two commuting algebra automorphisms, and the theory of Yetter–Drinfel’d modules serves as a fundamental tool to construct braided monoidal categories. However, the parameterized L-R-smash biproduct and associated Yetter–Drinfel’d-Long categorical constructions have not been developed in the BiHom setting. Motivated by the above observations, we introduce BiHom-L-R-\((\begin{smallmatrix}m,n,p,q\\s,t,u,v\end{smallmatrix})\)-smash biproducts, give sufficient conditions for them to be BiHom-bialgebras, and define BiHom-\((\begin{smallmatrix}m,n,p,q\\s,t,u,v\end{smallmatrix})\)-Yetter–Drinfel’d-Long bimodules. We prove that the bimodule category is a strict braided monoidal category, and it is braided monoidally equivalent to a BiHom-Yetter–Drinfel’d category for a finite-dimensional BiHom-Hopf algebra $H$. This equivalence is the main conceptual contribution of this paper. It extends the well-known reconstruction philosophy of Hopf algebras into the multi-parameter BiHom setting and offers a new way to produce braided categories associated with BiHom-Hopf algebras.

This paper is organized as follows:

In Section 2, we recall basic notions and constructions of BiHom-Hopf algebra. In Section 3, we first introduce the notions of BiHom-L-R-$(m, n,p,q)$-smash product algebra $D\n_{m, n,p,q}H$ and  BiHom-L-R-$(s,t,u,v)$-smash coproduct coalgebra $D\n^{s,t,u,v}H$, where $m,n,p,q,s,t,u,v\in \mathbb{Z}$. Then the sufficient condition for $D\n_{m, n,p,q}H$ and $D\n^{s,t,u,v}H$ on $D\o H$ to be a BiHom-bialgebra are derived (called BiHom-L-R-$\binom{m,n,p,q}{s,t,u,v}$-smash biproduct, denoted by $D\n H$), generalizing the Radford $(m, n)$-biproduct in \cite{MLC20} and L-R-$(m, n,p,q)$-smash biproduct in \cite{Y26}. In Section 4, we introduce the concept of BiHom-$\binom{m,n,p,q}{s,t,u,v}$-Yetter-Drinfel'd-Long bimodule, and prove the category $\mathcal{LR}(H)\binom{m,n,p,q}{s,t,u,v}$ of BiHom-$\binom{m,n,p,q}{s,t,u,v}$-Yetter-Drinfel'd-Long bimodules is a strict braided monoidal category. Moreover, we obtain that $D\n H$ is a BiHom-L-R-$\binom{m,n,p,q}{s,t,u,v}$-smash biproduct bialgebra if and only if $D$ is a BiHom-bialgebra in $\mathcal{LR}(H)\binom{m,n,p,q}{s,t,u,v}$ with an extra condition. In Section 5, let $H$ be a finite dimensional BiHom-Hopf algebra, we prove $\mathcal{LR}(H)\binom{m,n,p,q}{s,t,u,v}$ is  isomorphic to the category $\!^{H\o H^{*}}_{H\o H^{*}}\mathcal{YD}\binom{s,t}{p,q}$ of Yetter-Drinfel'd modules as braided monoidal categories.

\section{Preliminaries}
\def\theequation{2.\arabic{equation}}
\setcounter{equation} {0}
Throughout this paper, all algebraic systems are over a field $\B$.  We denote the identity map by $\mathrm{id}$. We shall use the sigma notation in the versions of Sweedler
for $\D:\D(h)=h_{1}\o h_{2}$. In order to facilitate our computations, we always omit the summation symbol $\sum$.

A monoidal category $\mathcal{C}=(\mathcal{C},\otimes,\mathcal{I},a,l,r)$ is a category $\mathcal{C}$ equipped with a tensor product functor $\otimes:\mathcal{C}\times\mathcal{C}\longrightarrow\mathcal{C}$, with a tensor unit object $\mathcal{I}\in\mathcal{C}$, with an associativity constraint isomorphism $a=a_{U,V,W}:(U\otimes V)\otimes W\longrightarrow U\otimes(V\otimes W)$ for any objects $U,V,W\in\mathcal{C}$, a left unit constraint $l=l_U:\mathcal{I}\otimes U\longrightarrow U$ and a right unit constraint $r=r_U:U\otimes\mathcal{I}\longrightarrow U$, for any object $U\in\mathcal{C}$, such that the pentagon axiom $a_{U,V,W\otimes X}\circ a_{U\otimes V,W,X}=(U\otimes a_{V,W,X})\circ a_{U,V\otimes W,X}\circ(a_{U,V,W}\otimes X)$ and the triangle axiom $(U\otimes l_V)\circ a_{U,\mathcal{I},V}=(r_U\otimes V)$ hold, for any objects $U,V,W,X\in\mathcal{C}$. A monoidal category $\mathcal{C}$ is strict if all the constraints are identities. 

A \emph{braiding}\cite{K95} of a monoidal category $\mathcal{C}$ is a family of natural isomorphisms $c=c_{V,W}:V\otimes W\longrightarrow W\otimes V$ such that the following conditions hold
\begin{align*}
\begin{cases}
&c_{U,V\otimes W}=a_{V,W,U}^{-1}\circ(V\otimes c_{U,W})\circ a_{V,U,W}\circ(c_{U,V}\otimes W)\circ a_{U,V,W}^{-1},\\
&c_{U\otimes V,W}=a_{W,U,V}\circ(c_{U,W}\otimes V)\circ a_{U,W,V}^{-1}\circ(U\otimes c_{V,W})\circ a_{U,V,W},
\end{cases}
\end{align*}
for any $U,V,W\in\mathcal{C}$, where $a$ is the associativity constraint in the monoidal category $\mathcal{C}$.

Note that a braided monoidal category is a monoidal category $\mathcal{C}$ with a braiding.

In what follows, we will recall from \cite{GMMP15}, \cite{LWL23}, \cite{ZWC24}  some information concerning BiHom-structures.

A \emph{unital BiHom-associative algebra $A$} is a 5-tuple $(A, \mu_A, 1_A, \alpha_A, \beta_A)$, in which $A$ is a linear space, $1_A \in A$ is an element (the unit), $\alpha_A, \beta_A: A \to A$ are linear isomorphisms, $\mu_A: A \otimes A \to A$ is a linear map with the notation $\mu_{A}(a \otimes b) = ab$, such that, for all $a,b,c \in A$:
\begin{align*}
\begin{cases}
&\alpha_A(1_A) = \beta_A(1_A) = 1_A,\quad a1_A = \alpha_A(a),\quad 1_Aa = \beta_A(a),\\
& \alpha_A(a)(bc) = (ab)\beta_A(c),\\
&\alpha_A \circ \beta_A = \beta_A \circ \alpha_A,\quad \alpha_A(ab) = \alpha_A(a)\alpha_A(b),\quad \beta_A(ab) = \beta_A(a)\beta_A(b).
\end{cases}
\end{align*}

In this paper, the algebras we mainly discussed are this kind of unital BiHom-associative algebras, and in the following we call them the \emph{BiHom-algebras}.

A $\emph{BiHom-algebra map}$ $f:(A, \mu_A, 1_A, \alpha_A, \beta_A)\lr (A', \mu_{A'}, 1_{A'}, \alpha_{A'}, \beta_{A'})$ is a map $f:A\lr A'$ such that $\a_{A'}\circ f=f\circ \a_{A}$, $\b_{A'}\circ f=f\circ \b_{A}$, $f(ab)=f(a)f(b)$ and $f(1_{A})=1_{A'}$,  for any $a,b\in A$.

A  \emph{counital BiHom-coassociative coalgebra $C$} is a 5-tuple $(C, \Delta_C, \varepsilon_C, \om_C, \psi_C)$, in which $C$ is a linear space, $\om_C, \psi_C: C \to C$ are linear isomorphisms, $\varepsilon_C: C \to \B$ and $\Delta_C: C \to C \otimes C$ are linear maps, such that, for all $c\in C$:
\begin{align*}
\begin{cases}
&\varepsilon_C(\om_C(c)) = \varepsilon_C(\psi_C(c)) = \varepsilon_C(c),\quad c_1\varepsilon_C(c_2) = \om_C(c),\quad \varepsilon_C(c_1)c_2 = \psi_C(c),\\
& \om_C(c_1) \otimes \Delta_C(c_2) = \Delta_C(c_1) \otimes \psi_C(c_2),\\
&\om_C \circ \psi_C = \psi_C \circ \om_C,\quad \Delta_C(\om_C(c)) = \om_C(c_1) \otimes \om_C(c_2),\quad \Delta_C(\psi_C(c)) = \psi_C(c_1) \otimes \psi_C(c_2).
\end{cases}
\end{align*}

Analogue to BiHom-algebras, \emph{BiHom-coalgebras} will be short for counital BiHom-coassociative coalgebra without any confusion.

A \emph{BiHom-coalgebras map $f:(C, \Delta_C, \varepsilon_C, \om_C, \psi_C)\lr (C', \Delta_{C'}, \varepsilon_{C'}, \om_{C'}, \psi_{C'})$} is a map $f:C\lr C'$ such that $\om_{C'}\circ f=f\circ \om_{C}$, $\psi_{C'}\circ f=f\circ \psi_{C}$, $\D_{C'}\circ f=(f\o f)\circ \D_{C}$ and $\varepsilon_{C'}\circ f=\varepsilon_{C}$.

A \emph{BiHom-bialgebra $H$} is a 9-tuple $(H,\mu,1,\Delta,\varepsilon,\a,\b,\om,\psi)$ with the property that $(H,\mu,1,\a,\b)$ is a BiHom-algebra, $(H,\Delta,\varepsilon,\om,\psi)$ is a BiHom-coalgebra, and $\Delta,\varepsilon$ are all morphisms of BiHom-algebras preserving unit, i.e., for all $h,g\in H$,
$$
\Delta(hg)=h_1g_1\otimes h_2g_2,\quad \v(hg)=\v(h)\v(g),\quad \Delta(1)=1\otimes 1,\quad \v(1)=1_\B.
$$
Moreover, $\a,\b$ are BiHom-coalgebra maps, $\om,\psi$ are BiHom-algebra maps, and they commute with each other.

\begin{example}
If $H = (H,\mu,1,\Delta,\varepsilon,\a,\b,\om,\psi)$ is a finite dimensional BiHom-bialgebra, $H^* = \mathrm{hom}(H, \B)$. Define the multiplication $\star$, the comultiplication $\Delta_{H^*}$ (with the notation $\Delta_{H^*}(p) = p_1 \otimes p_2$) and $\varepsilon_{H^*}$ by
$$
\begin{aligned}
(p \star q)(h) &= p(\alpha^{-1}\om^{-1}(h_1))q(\beta^{-1}\psi^{-1}(h_2)),\quad \varepsilon_{H^*}(p) = p(1_H),\\
(p_1 \otimes p_2)(h \otimes g) &= p(\alpha^{-1}\om^{-1}(h)\beta^{-1}\psi^{-1}(g)),
\end{aligned}
$$
where $p,q \in H^*$, $h,g \in H$. Define $\alpha_{H^*},\beta_{H^*},\om_{H^*},\psi_{H^*}$ by
$$
\alpha_{H^*}(p) = p \circ \a^{-1},\quad \beta_{H^*}(p) = p \circ \b^{-1},\quad \om_{H^*}(p) = p \circ \om^{-1},\quad \psi_{H^*}(p) = p \circ \psi^{-1}.
$$
Then $H^* = (H^*, \star, \varepsilon, \Delta_{H^*}, \varepsilon_{H^*}, \alpha_{H^*}, \beta_{H^*}, \om_{H^*}, \psi_{H^*})$ is a BiHom-bialgebra.
\end{example}

This example is slightly different from \cite{ZWC24} Ex. 2.6.

A \emph{BiHom-Hopf algebras} \cite{ZWC24}  is a BiHom-bialgebra $H=(H,\mu,1,\Delta,\varepsilon,\a,\b,\om,\psi)$ with a morphism (called the antipode)
$S\colon H\to H$ such that $S$ commutes with $\alpha,\beta,\om,\psi$, and satisfies, for any $h\in H$,
$$
h_1S(h_2)=S(h_1)h_2=\varepsilon(h)1_H.
$$

Note the definition of BiHom-Hopf algebras which is little different from \cite{GMMP15}, Def. 6.9.

\begin{proposition}\cite{ZWC24}
If $H$ is a BiHom-Hopf algebra, then

(1) the antipode $S$ satisfies 
\begin{align*}
&S(ab) = S\alpha^{-1}\beta(b)\,S\alpha\beta^{-1}(a),\quad S(1)=1,
\\&\Delta(S(a)) = S\om\psi^{-1}(a_2)\otimes S\om^{-1}\psi(a_1),\quad \varepsilon\circ S=\varepsilon,
\\&S\alpha^2\om^2=S\beta^2\psi^2;
\end{align*}

(2) if $S$ is a bijective map, then 
\begin{align*}
&\alpha^2\om^2=\beta^2\psi^2,\\
&S^{-1}(ab) = S^{-1}\alpha^{-1}\beta(b)S^{-1}\alpha\beta^{-1}(a),\quad S^{-1}(1_H)=1_H,\\
&\Delta(S^{-1}(a)) = S^{-1}\om\psi^{-1}(a_2)\otimes S^{-1}\om^{-1}\psi(a_1),\quad \varepsilon\circ S^{-1}=\varepsilon,\\
&S^{-1}\alpha^{-2}\beta^2(a_2)a_1 = a_2S^{-1}\alpha^2\beta^{-2}(a_1)=\varepsilon(a)1_H,\\
&S^{-1}\om^{2}\psi^{-2}(a_2)a_1 = a_2S^{-1}\om^{-2}\psi^2(a_1)=\varepsilon(a)1_H.
\end{align*}

\end{proposition}

Let $H$ be a BiHom-bialgebra. A \emph{left $H$-module} is a 6-tuple $(M,\tr_M,\alpha_M,\beta_M,\om_M,\psi_M)$, in which $M$ is a linear space, $\alpha_M,\beta_M,\om_M,\psi_M\colon M\to M$ are linear isomorphisms, $\tr \colon H\otimes M\to M : h\otimes m\m h\tr m$ is linear map, such that, for any $h,g\in H$, $m\in M$,
\begin{align*}
\begin{cases}
&\alpha_M,\beta_M,\om_M,\psi_M \text{ commute with each other},\\
&\alpha(h)\tr \alpha_M(m)=\alpha_M(h\tr m),\quad \beta(h)\tr \beta_M(m)=\beta_M(h\tr m),\\
&\om(h)\tr \om_M(m)=\om_M(h\tr m),\quad \psi(h)\tr \psi_M(m)=\psi_M(h\tr m),\\
&\alpha(h)\tr (g\tr m)=(hg)\tr \beta_M(m),\quad 1_H\tr m=\beta_M(m).
\end{cases}
\end{align*}

Similarly, we can define the \emph{right $H$-module}.

A \emph{left $H$-module map} $f:(M,\tr_M,\alpha_M,\beta_M,\om_M,\psi_M)\lr (N,\tr_N,\alpha_N,\beta_N,\om_N,\psi_N)$ is a map $f:M\lr N$ such that $\alpha_N\circ f=f\circ \alpha_M, \beta_N\circ f=f\circ \beta_M, \om_N\circ f=f\circ \om_M, \psi_N\circ f=f\circ \psi_M$ 
and $f\circ \tr_M=\tr_N\circ (\mathrm{id}_H\o  f)$. Moreover, if $(M,\tr_M,\alpha_M,\beta_M,\om_M,\psi_M)$ is both a left $H$-module (via action $\tr$) and right $H$-module (via action $\tl$) and satisfies
\begin{align*}
\a_{H}(a)\tr (m\tl b)=(a\tr m)\tl \b_{H}(b),
\end{align*}
for all $a,b\in H, m\in M$, we call $M$ is an \emph{$H$-bimodule}.

Let $H$ be a BiHom-bialgebra. A BiHom-algebra $D$ is called a \emph{left $H$-module BiHom-algebra} if $D$ is a left $H$-module together with the action $\tr$ such that, for any $a,b \in D$ and $h \in H$:
\begin{align*}
h\tr (ab) =(\alpha^{-1}\omega^{-1}(h_1) \tr a)(\beta^{-1}\psi^{-1}(h_2)\tr b), \quad h\tr 1_{D} =\v(h)1_{D}.
\end{align*}

Similarly, we have right $H$-module BiHom-algebra. Moreover, $D$ is called \emph{$H$-bimodule BiHom-algebra} if $D$ as an $H$-BiHom-bimodule is both left $H$-module BiHom-algebra and right $H$-module BiHom-algebra.

Let $H$ be a BiHom-bialgebra. A \emph{right $H$-comodule} is a 6-tuple $(M, \r_M, \alpha_M,\beta_M,\om_M,\psi_M)$, in which  $M$ is a linear space, $\alpha_M,\beta_M,\om_M,\psi_M\colon M\to M$ are linear isomorphisms, $\r\colon M\to M\otimes H : m \m m_{(0)}\otimes m_{(1)}$ is linear map, such that,  for any $m\in M$,
\begin{align*}
\begin{cases}
&\om_M,\psi_M,\alpha_M,\beta_M \text{ commute with each other},\\
&(\alpha_M\otimes\alpha_{H})\circ\r=\r\circ\alpha_M,\quad (\beta_M\otimes\beta_{H})\circ\r=\r\circ\beta_M,\\
&(\om_M\otimes\om_{H})\circ\r=\r\circ\om_M,\quad (\psi_M\otimes\psi_{H})\circ\r=\r\circ\psi_M,\\
&\om_M(m_{(0)})\otimes m_{{(1)1}}\otimes m_{{(1)2}}=m_{(0)(0)}\otimes m_{(0)(1)}\otimes\psi_{H}(m_{(1)}),\quad m_{(0)}\varepsilon(m_{(1)})=\om_M(m).
\end{cases}
\end{align*}

Similarly, we can define the \emph{left $H$-comodule}.

A \emph{right $H$-comodule map} $f:(M,\r_M,\alpha_M,\beta_M,\om_M,\psi_M)\lr (N,\r_N,\alpha_N,\beta_N,\om_N,\psi_N)$ is a map $f:M\lr N$ such that 
$\alpha_N\circ f=f\circ\alpha_M, \beta_N\circ f=f\circ\beta_M,\om_N\circ f=f\circ\om_M,
\psi_N\circ f=f\circ\psi_M, \r_N\circ f=(f\otimes\mathrm{id}_H)\circ\r_M$.
Moreover, if $(M,\r_M,\alpha_M,\beta_M,\om_M,\psi_M)$ is both a left $H$-comodule (via coaction $\r^{l}_M(m)=m_{[-1]}\o m_{[0]}$) and right $H$-comodule (via action $\r^{r}_M(m)=m_{(0)}\o m_{(1)}$) and satisfies
\begin{align*}
&\om_{H}(m_{[-1]})\o m_{[0](0)}\o m_{[0](1)}=m_{(0)[-1]}\o m_{(0)[0]}\o \psi_{H}(m_{(1)}),
\end{align*}
for all $m\in M$, we call $M$ is a \emph{$H$-bicomodule}.

Let $H$ be a BiHom-bialgebra. A BiHom-coalgebra $D$ is called a \emph{left $H$-comodule BiHom-coalgebra} if $D$ is a left $H$-comodule together with the coaction $\r$ such that, for any $c\in D$ and $h \in H$:
\begin{align*}
c_{[-1]}\o c_{[0]1}\o c_{[0]2}=\alpha_H^{-1}\omega_H^{-1}(c_{1[-1]})\beta_H^{-1}\psi_H^{-1}(c_{2[-1]})\o c_{1[0]}\o c_{2[0]}
, \quad c_{[-1]}\v_{C}(c_{[0]})=\v_{C}(c)1.
\end{align*}

Similarly, we have right $H$-comodule BiHom-coalgebra. Moreover, $D$ is called \emph{$H$-bicomodule BiHom-coalgebra} if $D$ as an $H$-bicomodule is both left $H$-comodule BiHom-coalgebra and right $H$-comodule BiHom-coalgebra.

\section{BiHom-L-R-$\binom{m,n,p,q}{s,t,u,v}$-smash biproduct}
\def\theequation{3.\arabic{equation}}
\setcounter{equation} {0}
In this section, we first introduce the notions of BiHom-L-R-$(m,n,p,q)$-smash product algebra $D\n_{m,n,p,q} H$ and Bihom-L-R-$(s,t,u,v)$-smash coproduct coalgebra $D\n^{s,t,u,v} H$, where $m,n,p,q,s,t,u,v\in \mathbb{Z}$. Then the sufficient condition for $D\n_{m,n,p,q} H$ and $D\n^{s,t,u,v} H$ on $D\o H$ to be a BiHom-bialgebra are derived (called BiHom-L-R-$\binom{m,n,p,q}{s,t,u,v}$-smash biproduct, denoted by $D\n H$).
\subsection{BiHom-L-R-$(m,n,p,q)$-smash product}
\begin{proposition}
Let $H$ be a BiHom-bialgebra, $D$ be an $H$-bimodule BiHom-algebra and $m,n,p,q\in \mathbb{Z}$. Then $(D\n_{m, n,p,q}H, \b\o \a)$ ($=D\o H$ as a vector space) is a BiHom-algebra with the unit $1_{D}\n 1_{H}$ and the multiplication
\begin{align*}
(a\n_{m, n,p,q} h)(b\n_{m, n,p,q} g)=(\alpha_D^{-1}(a)\tl \a^{m}\psi^{n}(g_2))(\b^{p}\om^{q}(h_1)\tr\beta_D^{-1}(b))\n \psi^{-1}(h_2)\omega^{-1}(g_1),
\end{align*}
for all $a, b\in D, h,g\in H$. In this case, we call $D\n_{m, n,p,q}H$ BiHom-L-R-$(m,n,p,q)$-smash product algebra.
\end{proposition}
\begin{proof}
For all $a, b,c\in D, h,g,k\in H$, first $(\a_{D}\o \a)((a\n h)(b\n g))=((\a_{D}\o \a)(a\n h))((\a_{D}\o \a)(b\n g)),(\b_{D}\o \b)((a\n h)(b\n g))=((\b_{D}\o \b)(a\n h))((\b_{D}\o \b)(b\n g)), (\a_{D}\o \a)\circ (\b_{D}\o \b)=(\b_{D}\o \b)\circ (\a_{D}\o \a)$.
Let $A=\a^{m}\psi^{n}, B=\b^{p}\om^{q}$. Then we have
\begin{align*}
&[(a\n h)(b\n  g)](\b_{D}(c)\n  \b(f))
\\&=[(\alpha_D^{-1}(a)\tl  A(g_2))( B(h_1)\tr\beta_D^{-1}(b))\n \psi^{-1}(h_2)\omega^{-1}(g_1)](\b_{D}(c)\n  \b (f))
\\&=(\alpha_D^{-1}((\alpha_D^{-1}(a)\tl  A(g_2))( B(h_1)\tr\beta_D^{-1}(b)))\tl  A\b (f_{2}))
\\&\quad ( B(\psi^{-1}(h_{21})\omega^{-1}(g_{11}))\tr c)
\n \psi^{-1}(\psi^{-1}(h_{22})\omega^{-1}(g_{12}))\b\omega^{-1} (f_{1})
\\&=(([\a_{D}^{-2}(a)\tl A\a^{-1} (g_{2})]\tl A\a^{-1}\b \omega^{-1}(f_{21}))[[ B\a^{-1}(h_{1})\tr\alpha_D^{-1}\beta_D^{-1}(b)]\tl  A\psi^{-1}(f_{22})])
\\&\quad[[ B\psi^{-1}(h_{21}) B\om^{-1}(g_{11})]\tr c]
\n [\psi^{-2}(h_{22})\omega^{-1}\psi^{-1}(g_{12})]\b\omega^{-1}(f_{1})
\\&=([\a_{D}^{-1}(a)\tl  A(g_{2})]\tl A\b \omega^{-1}(f_{21}))([[ B\a^{-1}(h_{1})\tr\alpha_D^{-1}\beta_D^{-1}(b)]\tl  A\psi^{-1}(f_{22})]
\\&\quad[[ B\beta^{-1}\psi^{-1}(h_{21})B\beta^{-1} \om^{-1}(g_{11})]\tr\beta_D^{-1}(c)])
\n [\psi^{-2}(h_{22})\omega^{-1}\psi^{-1}(g_{12})]\b\omega^{-1}(f_{1})
\\&=([\a_{D}^{-1}(a)\tl  A(g_{2})]\tl A\b \omega^{-1}(f_{21}))([ B(h_{1})\tr\beta_D^{-1}(\alpha_D^{-1}(b)\tl  A\psi^{-1}(f_{22}))]
\\&\quad[ B\a \beta^{-1}\psi^{-1}(h_{21})\tr\beta_D^{-1}( B\om^{-1}(g_{11})\tr\beta_D^{-1}(c))])
\n [\psi^{-2}(h_{22})\omega^{-1}\psi^{-1}(g_{12})]\b\omega^{-1}(f_{1})
\\&=([\a_{D}^{-1}(a)\tl  A\psi^{-1}(g_{22})]\tl A\b \omega^{-1}(f_{12}))([ B\omega^{-1}(h_{11})\tr\beta_D^{-1}(\alpha_D^{-1}(b)\tl  A(f_2))]
\\&\quad[ B\a \beta^{-1}\psi^{-1}(h_{12})\tr\beta_D^{-1}( B(g_1)\tr\beta_D^{-1}(c))])
\n [\psi^{-1}(h_2)\omega^{-1}\psi^{-1}(g_{21})]\b\omega^{-2}(f_{11})
\\&=(a\tl  A(\psi^{-1}(g_{22})\omega^{-1}(f_{12})))( B\a (h_1)\tr\beta_D^{-1}((\alpha_D^{-1}(b)\tl  A(f_2))( B(g_1)\tr\beta_D^{-1}(c))))
\\&\quad\n \a \psi^{-1}(h_2)\omega^{-1}(\psi^{-1}(g_{21})\omega^{-1}(f_{11}))
\\&=(\a_{D}(a)\n  \a (h))[(\alpha_D^{-1}(b)\tl  A(f_2))( B(g_1)\tr\beta_D^{-1}(c))\n \psi^{-1}(g_2)\omega^{-1}(f_1)]
\\&=(\a_{D}(a)\n  \a (h))[(b\n g)(c\n  f)]
\end{align*}
and
\begin{align*}
(a \n h)(1_{D}\n 1_{H})=\a_{D}(a)\n \a(h),(1_{D}\n 1_{H})(a\n h)=\b_{D}(a)\n \b(h).
\end{align*}
Finally $(\a_{D}\o \a)(1_{D}\n 1_{H})=1_{D}\n 1_{H}=(\b_{D}\o \b)(1_{D}\n 1_{H})$.

The proof is completed.
\end{proof}
\subsection{L-R-$(s,t,u,v)$-smash coproduct}
\begin{proposition}
Let $H$ be a BiHom-bialgebra, $D$ be an $H$-bicomodule BiHom-coalgebra and $s,t,u,v\in \mathbb{Z}$. Then $D\n^{s,t,u,v}H$ ($=D\o H$ as a vector space) is a BiHom-coalgebra with the counit $\v_{D}\n \v_{H}$ and the comultiplication
\begin{align*}
\D_{D\n H}(c\n^{s,t,u,v} h)=\om_D^{-1}(c_{1(0)})\n\a^{s}\psi^{t}(c_{2[-1]})\b^{-1}(h_1)\o \psi_D^{-1}(c_{2[0]})\n\alpha^{-1}( h_{2}) \b^{u}\om^{v}(c_{1(1)}),
\end{align*}
for all $c\in D, h\in H$. In this case, we call $(D\sharp^{(s,t,u,v)}H, \b\o \a)$ BiHom-L-R-$(s,t,u,v)$-smash coproduct coalgebra.
\end{proposition}
\begin{proof}
For all $c\in D,h\in H$, first $\D_{D\n H}\circ (\om_{D}\o \om)=(\om_{D}\o \om\o \om_{D}\o \om)\circ \D_{D\n H}, \D_{D\n H}\circ (\psi_{D}\o \psi)=(\psi_{D}\o \psi\o \psi_{D}\o \psi)\circ \D_{D\n H}, (\om_{D}\o \om)\circ (\psi_{D}\o \psi)=(\psi_{D}\o \psi)\circ (\om_{D}\o \om)$. Let $C=\a^{s}\psi^{t}, D=\b^{u}\om^{v}$. Then
\begin{align*}
&(\D_{D\n H}\o \psi_D\o \psi )\circ \D_{D\n H}
\\&=(\D_{D\n H}\o \psi_D\o \psi )[\om^{-1}(c_{1(0)})\n C(c_{2[-1]})\b ^{-1}(h_1)\o \psi_D^{-1}(c_{2[0]})\n\alpha ^{-1}( h_{2}) D(c_{1(1)})]
\\&=\om_D^{-2}(c_{1(0)1(0)})\n C\om ^{-1}(c_{1(0)2[-1]})[ C\b ^{-1}(c_{2[-1]1})\b ^{-2}(h_{11})]
\o \om_D^{-1}\psi_D^{-1}(c_{1(0)2[0]})\\&\quad\n [ C\a^{-1}(c_{2[-1]2})\a^{-1}\b ^{-1}(h_{12})]  D\om ^{-1}(c_{1(0)1(1)})
\o c_{2[0]}\n \alpha ^{-1}\psi (h_{2})D\psi (c_{1(1)})
\\&=\om_D^{-2}(c_{11(0)(0)})\n C\omega ^{-1}(c_{12(0)[-1]})[ C\beta ^{-1}(c_{2[-1]1})\b ^{-2}(h_{11})]
\\&\quad\o \om_D^{-1}\psi_D^{-1}(c_{12(0)[0]})\n[ C\a^{-1}(c_{2[-1]2})\alpha ^{-1}\b ^{-1}( h_{12})] D\om^{-1}(c_{11(0)(1)})
\\&\quad\o c_{2[0]}\n\alpha ^{-1}\psi( h_{2}) [D\alpha ^{-1}\om^{-1}\psi(c_{11(1)}) D\beta ^{-1}(c_{12(1)})]
\\&=\om_D^{-2}(c_{11(0)(0)})\n C\om(\alpha ^{-1}\omega ^{-2}(c_{12(0)[-1]})\beta ^{-1}\om^{-1}(c_{2[-1]1}))\b ^{-1}(h_{11})
\\&\quad\o \om_D^{-1}\psi_D^{-1}(c_{12(0)[0]})\n C(c_{2[-1]2})[\alpha ^{-1}\b ^{-1}( h_{12}) D\b ^{-1}\om^{-1}(c_{11(0)(1)})]
\\&\quad\o c_{2[0]}\n[\alpha ^{-2}\psi( h_{2}) D\alpha ^{-1}\om^{-1}\psi(c_{11(1)})] D(c_{12(1)})
\\&=\om_D^{-1}(c_{11(0)})\n C\om(\alpha ^{-1}\omega ^{-1}(c_{12[-1]})\beta ^{-1}(c_{2[-1]}))\b ^{-1}(h_{11})
\\&\quad\o \om_D^{-1}\psi_D^{-1}(c_{12[0](0)})\n C(c_{2[0][-1]})[\alpha ^{-1}\b ^{-1}( h_{12}) D\b ^{-1}\om^{-1}(c_{11(1)1})]
\\&\quad\o \psi_D^{-1}(c_{2[0][0]})\n[\alpha ^{-2}\psi( h_{2}) D\alpha ^{-1}\om^{-1}(c_{11(1)2})] D\psi^{-1}(c_{12[0](1)})
\\&=c_{1(0)}\n C\om(\alpha ^{-1}\omega ^{-1}(c_{21[-1]})\beta ^{-1}\psi ^{-1}(c_{22[-1]}))\b ^{-1}\om(h_1)
\\&\quad\o \om_D^{-1}\psi_D^{-1}(c_{21[0](0)})\n C\psi^{-1}(c_{22[0][-1]})[\alpha ^{-1}\b ^{-1}( h_{21}) D\b ^{-1}(c_{1(1)1})]
\\&\quad\o \psi_D^{-2}(c_{22[0][0]})\n[\alpha ^{-2}( h_{22}) D\alpha ^{-1}(c_{1(1)2})] D\psi^{-1}(c_{21[0](1)})
\\&=c_{1(0)}\n C\om(c_{2[-1]})\b ^{-1}\om(h_1)
\\&\quad\o \om_D^{-1}\psi_D^{-1}(c_{2[0]1(0)})\n C\psi^{-1}(c_{2[0]2[-1]})[\alpha ^{-1}\b ^{-1}( h_{21}) D\b ^{-1}(c_{1(1)1})]
\\&\quad\o \psi_D^{-2}(c_{2[0]2[0]})\n[\alpha ^{-2}( h_{22}) D\alpha ^{-1}(c_{1(1)2})] D\psi^{-1}(c_{2[0]1(1)})
\\&=(\om_D\o \om \o \D_{D\n H})[\om_D^{-1}(c_{1(0)})\n C(c_{2[-1]})\b ^{-1}(h_1)\o \psi^{-1}(c_{2[0]})\n\alpha ^{-1}( h_{2}) D(c_{1(1)})]
\\&=(\om_D\o \om \o \D_{D\n H})\circ \D_{D\n H}
\end{align*}
and
\begin{align*}
(\mathrm{id}_{D\n H}\o \v_{D\n H})\circ \D_{D\n H}(a\n h)=\om_{D}(a)\n \om (h), (\v_{D\n H}\o& \mathrm{id}_{D\n H})\circ \D_{D\n H}(a\n h)=\psi_{D}(a)\n \psi(h).
\end{align*}
Finally, $\v_{D\n H}\circ (\om_{D}\o \om)=\v_{D\n H}=\v_{D\n H}\circ (\psi_{D}\o \psi)$.

The proof is completed.
\end{proof}
\subsection{BiHom-L-R-$\binom{m,n,p,q}{s,t,u,v}$-smash biproduct}
\begin{definition}
Let $H$ be a BiHom-bialgebra, and let $D$ be an $H$-bimodule BiHom-algebra and an $H$-bicomodule BiHom-coalgebra. Let $\a_{D},\b_{D},\om_{D},\psi_{D}\in Aut(D)$, where $\a_{D},\b_{D}$ are BiHom-coalgebra maps, $\om_{D},\psi_{D}$ are BiHom-algebra maps, and they commute with each other. For all $h\in H, a,b\in D$, consider the following list of conditions:
\begin{align}
&\v_{D}(1_{D})=1,\quad \v_{D}(ab)=\v_{D}(a)\v_{D}(b),
\\&\v_{D}(h\tr a)=\v_{D}(a\tl h)=\v_{H}(h)\v_{D}(a),
\\&\r_{R}(1_{D})=1_{D}\o 1_{H}, \quad \r_{L}(1_{D})=1_{H}\o 1_{D},
\\&\r_{R}(ab)=\r_{R}(a)\r_{R}(b),\quad \r_{L}(ab)=\r_{L}(a)\r_{L}(b),
\\&\D_{D}(1_{D})=1_{D}\o 1_{D},
\\&\D_{D}(h\tr a)=h_{1}\tr a_{1}\o h_{2}\tr a_{2},\quad \D_{D}(a\tl h)=a_{1}\tl h_{1}\o a_{2}\tl h_{2},
\\&\Delta_D(ab) = a_{1}(\a^{s+1}\b^{p}\om^{p+1}\psi^{t}(a_{2[-1]})\tr\beta_D^{-1}\om_D^{-1}(b_{1(0)}))\notag 
\\&\qquad\qquad\quad \o (\alpha_D^{-1}\psi_D^{-1}(a_{2[0]})\tl \a^{m}\b^{u+1}\om^{v}\psi^{n+1}(b_{1(1)}))b_{2}
\\&(\a^{s}\b^{p+1}\om^{q}\psi^{t+1}(h_1)\tr b)_{[-1]}\a\b_H^{-1}\psi_H^{-1}\om(h_{2})\o (\a^{s}\b^{p+1}\om^{q}\psi^{t+1}(h_1)\tr b)_{[0]}\notag
\\&\qquad\qquad\quad=h_{1}\a_H(b_{[-1]})\o \a^{s}\b^{p+1}\om^{q+1}\psi^{t+1}( h_{2}) \tr b_{[0]}
\\&(h\tr b)_{(0)}\o(h\tr b)_{(1)}=(\om_{H}(h)\tr b_{(0)})\o \b_{H}(b_{(1)})
\\&[a\tl \a^{m+1}\b^{u}\om^{v+1}\psi^{n}(g_2)]_{(0)}\o\alpha_H^{-1}\b_H\psi_H\omega_H^{-1}(g_{1})[a\tl \a^{m+1}\b^{u}\om^{v+1}\psi^{n}(g_2)]_{(1)}\notag
\\&\qquad\qquad\quad=a_{(0)}\tl \a^{m+1}\b^{u}\om^{v+1}\psi^{n+1}(g_{1})\o\b_H(a_{(1)})g_{2}
\\&([a\tl g]_{[-1]})\o [a\tl g]_{[0]}=\a_{H}(a_{[-1]})\o a_{[0]}\tl \psi_{H}( g)
\\&\a_A^{-1}\om_A^{-1}(a_{(0)})\tl \a^{m+s+1}\om\psi^{n+t}(b_{[-1]})\o \b^{p+u+1}\om^{q+v}\psi(a_{(1)})\tr \beta_A^{-1}\psi_A^{-1}(b_{[0]})=a\o b\label{e3.12}
\end{align}
If all these conditions hold, we will call $(H,D)$ an BiHom-L-R-$\binom{m,n,p,q}{s,t,u,v}$admissible pair.
\end{definition}
\begin{theorem}
If $(H,D)$ is an L-R-$\binom{m,n,p,q}{s,t,u,v}$-admissible pair, then $D\n H$ under BiHom-L-R-$(m, n,p,q)$-smash product and BiHom-L-R-$(s,t,u,v)$-smash coproduct is a BiHom-bialgebra, we will call $D\n H$ BiHom-L-R-$\binom{m,n,p,q}{s,t,u,v}$-smash biproduct bialgebra.
\end{theorem}
\begin{proof}
Let $A=\a^{m}\psi^{n},B=\b^{p}\om^{q},C=\a^{s}\psi^{t},D=\b^{u}\om^{v}$, for all $a, b\in D, h, g\in H$, on one hand,
\begin{align*}
&\Delta(a\n h)\D(b\n g)
\\&=(\om_D^{-1}(a_{1(0)})\n C(a_{2[-1]})\b ^{-1}(h_1)\o \psi_D^{-1}(a_{2[0]})\n\alpha ^{-1}( h_{2})  D(a_{1(1)}))
\\&\quad(\om_D^{-1}(b_{1(0)})\n C(b_{2[-1]})\b ^{-1}(g_1)\o \psi_D^{-1}(b_{2[0]})\n\alpha ^{-1}( g_{2})  D(b_{1(1)}))
\\&=(\alpha_D^{-1}\om_D^{-1}(a_{1(0)})\tl  A[ C(b_{2[-1]2})\b ^{-1}(g_{12})])
( B[ C(a_{2[-1]1})\b ^{-1}(h_{11})]\tr\beta_D^{-1}\om_D^{-1}(b_{1(0)}))
\\&\quad\n\psi ^{-1}[ C(a_{2[-1]2})\b ^{-1}(h_{12})]\omega ^{-1}[ C(b_{2[-1]1})\b ^{-1}(g_{11})]
\\&\quad\o (\alpha_D^{-1}\psi_D^{-1}(a_{2[0]})\tl  A[\alpha ^{-1}( g_{22})  D(b_{1(1)2})])
( B[\alpha ^{-1}( h_{21})  D(a_{1(1)1})]\tr\beta_D^{-1}\psi_D^{-1}(b_{2[0]}))
\\&\quad\n\psi ^{-1}[\alpha ^{-1}( h_{22})  D(a_{1(1)2})]\omega ^{-1}[\alpha ^{-1}( g_{21})  D(b_{1(1)1})]
\\&=(\alpha_D^{-1}\om_D^{-2}(a_{1(0)(0)})\tl  A[ C(b_{2[0][-1]})\b ^{-1}(g_{12})])
( B[ C\om(a_{2[-1]})\b ^{-1}(h_{11})]\tr\beta_D^{-1}\om_D^{-2}(b_{1(0)(0)}))
\\&\quad\n\psi ^{-1}[ C(a_{2[0][-1]})\b ^{-1}(h_{12})]\omega ^{-1}[ C\om(b_{2[-1]})\b ^{-1}(g_{11})]
\\&\quad\o (\alpha_D^{-1}\psi_D^{-2}(a_{2[0][0]})\tl  A[\alpha ^{-1}( g_{22})  D\psi(b_{1(1)})])
( B[\alpha ^{-1}( h_{21})  D(a_{1(0)(1)})]\tr\beta_D^{-1}\psi_D^{-2}(b_{2[0][0]}))
\\&\quad\n\psi ^{-1}[\alpha ^{-1}( h_{22})  D\psi(a_{1(1)})]\omega ^{-1}[\alpha ^{-1}( g_{21})  D(b_{1(0)(1)})]
\\&=([\alpha_D^{-2}\om_D^{-2}(a_{1(0)(0)})\tl A C(b_{2[0][-1]})]\tl  A(g_{12}))
( B[ C\om(a_{2[-1]})\b ^{-1}(h_{11})]\tr\beta_D^{-1}\om_D^{-2}(b_{1(0)(0)}))
\\&\quad\n C\a \psi ^{-1}(a_{2[0][-1]})[[\a ^{-1}\b ^{-1}\psi ^{-1}(h_{12}) C\b ^{-1}(b_{2[-1]})]\b ^{-1}\om ^{-1}(g_{11})]
\\&\quad\o (\alpha_D^{-1}\psi_D^{-2}(a_{2[0][0]})\tl  A[\alpha ^{-1}( g_{22})  D\psi(b_{1(1)})])
( B( h_{21}) \tr[ B D(a_{1(0)(1)})\tr\beta_D^{-2}\psi_D^{-2}(b_{2[0][0]})])
\\&\quad\n[\alpha ^{-1}\psi ^{-1}(h_{22})[ D\alpha ^{-1}(a_{1(1)})\a ^{-1}\b ^{-1}\om ^{-1}(g_{21})]] D\beta \om^{-1}(b_{1(0)(1)})
\\&=\om_D^{-1}((\alpha_D^{-1}(a_{1(0)})\tl  A\om (g_{12}))
( B C\om ^{2}(a_{2[-1]}) B\b ^{-1}\om (h_{11})\tr\beta_D^{-1}\om_D^{-1}(b_{1(0)(0)})))
\\&\quad\n C\a \psi ^{-1}(a_{2[0][-1]})[[\a ^{-1}\b ^{-1}\psi ^{-1}(h_{12}) C\b ^{-1}(b_{2[-1]})]\b ^{-1}\om ^{-1}(g_{11})]
\\&\quad\o \psi_D^{-1}((\alpha_D^{-1}\psi_D^{-1}(a_{2[0][0]})\tl  A\alpha^{-1}\psi (g_{22})  A D\psi ^{2}(b_{1(1)}))
( B\psi ( h_{21}) \tr \b_D^{-1}(b_{2[0]})))
\\&\quad\n[\alpha ^{-1}\psi ^{-1}(h_{22})[ D\alpha ^{-1}(a_{1(1)})\a ^{-1}\b ^{-1}\om ^{-1}(g_{21})]] D\beta \om^{-1}(b_{1(0)(1)})
\end{align*}
On the other hand,
\begin{align*}
&\Delta[(a\n h)(b\n g)]
\\&=\Delta[(\alpha_D^{-1}(a)\tl A(g_2))(B(h_1)\tr\beta_D^{-1}(b))\n\psi^{-1}(h_2)\omega^{-1}(g_1)]
\\&=\om_D^{-1}([(\alpha_D^{-1}(a)\tl A(g_2))(B(h_1)\tr\beta_D^{-1}(b))]_{1(0)})
\\&\quad\n C([(\alpha_D^{-1}(a)\tl A(g_2))(B(h_1)\tr\beta_D^{-1}(b))]_{2[-1]})\b^{-1}[\psi^{-1}(h_{21})\omega^{-1}(g_{11})]
\\&\quad\o \psi_D^{-1}([(\alpha_D^{-1}(a)\tl A(g_2))(B(h_1)\tr\beta_D^{-1}(b))]_{2[0]})
\\&\quad\n\alpha^{-1}[\psi^{-1}(h_{22})\omega^{-1}(g_{12})]  D([(\alpha_D^{-1}(a)\tl A(g_2))(B(h_1)\tr\beta_D^{-1}(b))]_{1(1)})
\\&=\om_D^{-1}([(\a_D^{-1}(a_{1})\tl A(g_{21}))(B C\a\om((\a_D^{-1}(a_{2})\tl A(g_{22}))_{[-1]})
\\&\quad\tr\beta_D^{-1}\om_D^{-1}((B(h_{11})\tr\b_D^{-1}(b_{1}))_{(0)}))]_{(0)})
\\&\quad\n C([(\alpha_D^{-1}\psi_D^{-1}((\a_D^{-1}(a_{2})\tl A(g_{22}))_{[0]})\tl A D\b\psi((B(h_{11})\tr\b_D^{-1}(b_{1}))_{(1)}))
\\&\quad(B(h_{12})\tr\b_D^{-1}(b_{2}))]_{[-1]})\b^{-1}[\psi^{-1}(h_{21})\omega^{-1}(g_{11})]
\\&\quad\o \psi_D^{-1}([(\alpha_D^{-1}\psi_D^{-1}((\a_D^{-1}(a_{2})\tl A(g_{22}))_{[0]})
\\&\quad\tl A D\b\psi((B(h_{11})\tr\b_D^{-1}(b_{1}))_{(1)}))(B(h_{12})\tr\b_D^{-1}(b_{2}))]_{[0]})
\\&\quad\n\alpha^{-1}[\psi^{-1}(h_{22})\omega^{-1}(g_{12})]  D([(\a_D^{-1}(a_{1})\tl A(g_{21}))
(B C\a\om((\a_D^{-1}(a_{2})\tl A(g_{22}))_{[-1]})\tr
\\&\quad\beta_D^{-1}\om_D^{-1}((B(h_{11})\tr\b_D^{-1}(b_{1}))_{(0)}))]_{(1)})
\\&=\om_D^{-1}([(\a_D^{-1}(a_{1})\tl A(g_{21}))(B C\a\om(a_{2[-1]})
\tr\beta_D^{-1}\om_D^{-1}(B\omega(h_{11})\tr\b_D^{-1}(b_{1(0)})))]_{(0)})
\\&\quad\n C([(\alpha_D^{-1}\psi_D^{-1}(\a_D^{-1}(a_{2[0]})\tl A\psi(g_{22}))\tl A D\b\psi(b_{1(1)}))
\\&\quad(B(h_{12})\tr\b_D^{-1}(b_{2}))]_{[-1]})\b^{-1}[\psi^{-1}(h_{21})\omega^{-1}(g_{11})]
\\&\quad\o \psi_D^{-1}([(\alpha_D^{-1}\psi_D^{-1}(\a_D^{-1}(a_{2[0]})\tl A\psi(g_{22}))\tl A D\b\psi(b_{1(1)}))
(B(h_{12})\tr\b_D^{-1}(b_{2}))]_{[0]})
\\&\quad\n\alpha^{-1}[\psi^{-1}(h_{22})\omega^{-1}(g_{12})]  D([(\a_D^{-1}(a_{1})\tl A(g_{21}))
\\&\quad(B C\a\om(a_{2[-1]})\tr\beta_D^{-1}\om_D^{-1}((B\omega(h_{11})\tr\b_D^{-1}(b_{1(0)}))))]_{(1)})
\\&=\om_D^{-1}((\a_D^{-1}(a_{1})\tl A(g_{21}))_{(0)}(B C\om(a_{2[-1]})
B\b^{-1}(h_{11})\tr\beta_D^{-1}\om_D^{-1}(b_{1(0)}))_{(0)})
\\&\quad\n C((\a_D^{-1}\psi_D^{-1}(a_{2[0]})\tl A\a^{-1}(g_{22}) A D\psi(b_{1(1)}))_{[-1]}
\\&\quad(B(h_{12})\tr\b_D^{-1}(b_{2}))_{[-1]})\b^{-1}[\psi^{-1}(h_{21})\omega^{-1}(g_{11})]
\\&\quad\o \psi_D^{-1}((\alpha_D^{-1}\psi_D^{-1}(a_{2[0]})\tl A\alpha^{-1}(g_{22}) A D\psi(b_{1(1)}))_{[0]}
(B(h_{12})\tr\b_D^{-1}(b_{2}))_{[0]})
\\&\quad\n\alpha^{-1}[\psi^{-1}(h_{22})\omega^{-1}(g_{12})]  D((\a_D^{-1}(a_{1})\tl A(g_{21}))_{(1)}
\\&\quad(B C\om(a_{2[-1]})B\beta^{-1}(h_{11})\tr\beta_D^{-1}\om_D^{-1}(b_{1(0)}))_{(1)})
\\&=\om_D^{-1}((\a_D^{-1}(a_{1})\tl A(g_{21}))_{(0)}
(B C\om^{2}(a_{2[-1]})B\b^{-1}\om(h_{11})\tr\beta_D^{-1}\om_D^{-1}(b_{1(0)(0)})))
\\&\quad\n C\a\psi^{-1}(a_{2[0][-1]})[[ C\a^{-1}(B(h_{12})\tr\b_D^{-1}(b_{2}))_{[-1]}\b^{-2}\psi^{-1}(h_{21})]\b^{-1}\omega^{-1}(g_{11})]
\\&\quad\o \psi_D^{-1}((\alpha_D^{-1}\psi_D^{-1}(a_{2[0][0]})\tl A\alpha^{-1}\psi(g_{22}) A D\psi^{2}(b_{1(1)}))
(B(h_{12})\tr\b_D^{-1}(b_{2}))_{[0]})
\\&\quad\n[\alpha^{-1}\psi^{-1}(h_{22})[\alpha^{-2}\omega^{-1}(g_{12})  D\beta^{-1}(\a_D^{-1}(a_{1})\tl A(g_{21}))_{(1)}]]
D\beta\om^{-1}(b_{1(0)(1)})
\\&=\om_D^{-1}((\a_D^{-1}(a_{1})\tl A\psi^{-1}(g_{212}))_{(0)}
(B C\om^{2}(a_{2[-1]})B\b^{-1}\om^{2}(h_{1})\tr\beta_D^{-1}\om_D^{-1}(b_{1(0)(0)})))
\\&\quad\n C\a\psi^{-1}(a_{2[0][-1]})[[ C\a^{-1}(B\om^{-1}(h_{211})\tr\b_D^{-1}(b_{2}))_{[-1]}\b^{-2}\psi^{-2}(h_{212})]\b^{-1}(g_{1})]
\\&\quad\o \psi_D^{-1}((\alpha_D^{-1}\psi_D^{-1}(a_{2[0][0]})\tl A\alpha^{-1}\psi(g_{22}) A D\psi^{2}(b_{1(1)}))
(B\om^{-1}(h_{211})\tr\b_D^{-1}(b_{2}))_{[0]})
\\&\quad\n[\alpha^{-1}\psi^{-1}(h_{22})[\alpha^{-2}\omega^{-2}(g_{211})  D\beta^{-1}(\a_D^{-1}(a_{1})\tl A\psi^{-1}(g_{212}))_{(1)}]]
D\beta\om^{-1}(b_{1(0)(1)})
\\&=\om_D^{-1}((\alpha_D^{-1}(a_{1(0)})\tl A(g_{211}))
(B C\om^{2}(a_{2[-1]})B\b^{-1}\om^{2}(h_{1})\tr\beta_D^{-1}\om_D^{-1}(b_{1(0)(0)})))
\\&\quad\n C\a\psi^{-1}(a_{2[0][-1]})[[\a^{-1}\b^{-1}\om^{-1}\psi^{-1}(h_{211}) C\b^{-1}(b_{2[-1]})]\b^{-1}(g_{1})]
\\&\quad\o \psi_D^{-1}((\alpha_D^{-1}\psi_D^{-1}(a_{2[0][0]})\tl A\alpha^{-1}\psi(g_{22}) A D\psi^{2}(b_{1(1)}))
(B( h_{212}) \tr \b_D^{-1}(b_{2[0]})))
\\&\quad\n[\alpha^{-1}\psi^{-1}(h_{22})[ D\alpha^{-1}(a_{1(1)})\a^{-1}\b^{-1}\om^{-1}\psi^{-1}(g_{212})]]D\beta\om^{-1}(b_{1(0)(1)})
\\&=\om_D^{-1}((\alpha_D^{-1}(a_{1(0)})\tl A\om(g_{12}))
(B C\om^{2}(a_{2[-1]})B\b^{-1}\om(h_{11})\tr\beta_D^{-1}\om_D^{-1}(b_{1(0)(0)})))
\\&\quad\n C\a\psi^{-1}(a_{2[0][-1]})[[\a^{-1}\b^{-1}\psi^{-1}(h_{12}) C\b^{-1}(b_{2[-1]})]\b^{-1}\om^{-1}(g_{11})]
\\&\quad\o \psi_D^{-1}((\alpha_D^{-1}\psi_D^{-1}(a_{2[0][0]})\tl A\alpha^{-1}\psi(g_{22}) A D\psi^{2}(b_{1(1)}))
(B\psi( h_{21}) \tr \b_D^{-1}(b_{2[0]})))
\\&\quad\n[\alpha^{-1}\psi^{-1}(h_{22})[ D\alpha^{-1}(a_{1(1)})\a^{-1}\b^{-1}\om^{-1}(g_{21})]]D\beta\om^{-1}(b_{1(0)(1)})
\end{align*} 
That is, $\D_{D\n H}$ is a BiHom-algebra map. It is straightforward to verify that $\v_{D\n H}$ is also a BiHom-algebra map. Thus $D\n H$ is a BiHom-bialgebra.
\end{proof}
\begin{remark}
(1) If $H$ is a Hom-bialgebra, i.e., $\alpha= \beta= \psi = \omega$, and $m+n=p+q=s+t=u+v=-2$, BiHom-L-R-$\binom{m,n,p,q}{s,t,u,v}$-smash biproduct $D\n H$ becomes Hom-L-R-smash biproduct, which was introduced in \cite{LZ18}. 

(2) If $H$ is a monoidal Hom-bialgebra, i.e., $\psi= \omega =\alpha^{-1} = \beta^{-1}$, and $m-n=\hat{m},p-q=\hat{n},s-t=\hat{p},u-v=\hat{q}$, BiHom-L-R-$\binom{m,n,p,q}{s,t,u,v}$-smash biproduct $D\n H$ becomes L-R-$(\hat{m},\hat{n},\hat{p},\hat{q})$-smash biproduct, which was introduced in \cite{Y26}. 

(3) If $H$ is a bialgebra, i.e., $\alpha= \beta= \psi = \omega=\mathrm{id}$, BiHom-L-R-$\binom{m,n,p,q}{s,t,u,v}$-smash biproduct $D\n H$ becomes L-R-smash biproduct, which was introduced in \cite{PV10}.

(4) If the right action and right coaction are trivial, BiHom-L-R-$\binom{m,n,p,q}{s,t,u,v}$-smash biproduct $D\n H$ becomes BiHom-Radford's $\binom{p,q}{s,t}$-biproduct, the multiplication and comultiplication of $D\n H$ becomes, for all $a,b\in D, h,g\in H$,
\begin{align*}
&(a\sharp h)(b\sharp g)=a(\b^{p}\om^{q}(h_1)\tr\beta_D^{-1}(b))\sharp\psi^{-1}(h_2)g,
\\&\Delta(c\natural h)=c_{1}\n\a^{s}\psi^{t}(c_{2[-1]})\b^{-1}(h_1)\o \psi_D^{-1}(c_{2[0]})\n h_{2}.
\end{align*} 

(5) If $H$ is a Hom-bialgebra, i.e., $\alpha= \beta= \psi = \omega$, and $\a^{p+q}=\theta, \a^{s+t}=\om$, BiHom-Radford's $\binom{p,q}{s,t}$-biproduct $D\n H$ becomes $(\theta, \om)$-twisted Radford's Hom-biproduct, which was introduced in \cite{FK20}. 

(6) If $H$ is a monoidal Hom-bialgebra, i.e., $\psi= \omega =\alpha^{-1} = \beta^{-1}$, and $p-q=m,s-t=n$, BiHom-Radford's $\binom{p,q}{s,t}$-biproduct $D\n H$ becomes Radford's $(m,n)$ biproduct, which was introduced in \cite{MLC20}. 

\end{remark}
\begin{example}
Let \(D=\operatorname{sp}\{1_D,x\}\) over \(\B\) with \(\operatorname{char}\B\neq 2\), for all \(0\neq a,b,c,d\in\B\). 
Define the automorphism $\a_{D},\b_{D},\om_{D},\psi_{D}:D\to D:$
\begin{align*}
& \a_{D}(1_D)=\beta_{D}(1_D)=\om_{D}(1_D)=\psi_{D}(1_D)=1_D, 
\\&\a_{D}(x)=ax,\quad \beta_{D}(x)=bx,\quad  \om_{D}(x)=cx,\quad \psi_{D}(x)=dx.
\end{align*} 
Define the BiHom-algebra structure by
\begin{align*}
1_D1_D=1_D,\quad x1_D=\a_{D}(x)=ax,\quad 1_Dx=\beta_{D}(x)=bx,\quad  x^2=0.
\end{align*} 
Define the BiHom-coalgebra structure by
\begin{align*}
&\Delta(1_D)=1_D\otimes 1_D,\quad \varepsilon(1_D)=1_{\B},
\\&\Delta(x)=cx\otimes 1_D + 1_D\otimes dx,\quad \varepsilon(x)=0.
\end{align*}
 
Let \(H=\operatorname{sp}\{1_H,g\}\) be the group algebra and \(\a,\b,\om,\psi\) an automorphism of H. Then \((H,\c,1_{H},\D,\v,\a,\b,\om,\psi)\) with the following structure map is a BiHom-bialgebra:
\begin{align*}
&\a(1_H)=\beta(1_H)=\om(1_H)=\psi(1_H)=1_H, 
\\&1_H1_H=1_H,\quad g\cdot 1_H=\a(g),\quad 1_H\c g =\b(g),\quad  g\cdot g=1_H,
\\&\Delta(1_H)=1_H\otimes 1_H,\quad \varepsilon(1_H)=1_{\B},
\\&\Delta(g)=\om(g)\otimes\psi(g),\quad \varepsilon(g)=1_{\B}.
\end{align*} 

Define the left action \(\triangleright:H\otimes D\to D\) of H on D by
\begin{align*}
g\triangleright 1_D=1_D,\quad g\triangleright x=bx.
\end{align*} 
It is easy to see that $D$ is a left \(H\)-module BiHom-algebra. Define the right action \(\triangleleft:D\otimes H\to D\) of H on D by
\begin{align*}
1_D\triangleleft g=1_D,\quad x\triangleleft g=ax.
\end{align*} 
It is easy to see that $D$ is a right \(H\)-module BiHom-algebra. Moreover, $D$ is an $H$-bimodule. Define the left coaction \(\rho_L:D\to H\otimes D\) on D by
\begin{align*}
\rho_L(1_D)=1_H\otimes 1_D,\quad \rho_L(x)=\om(g)\otimes \psi_{D}(x).
\end{align*} 
It is easy to see that $D$ is a left \(H\)-comodule BiHom-coalgebra. Define the right coaction \(\rho_R:D\to D\otimes H\) on D by
\begin{align*}
\rho_R(1_D)=1_D\otimes 1_H,\quad \rho_R(x)=\om_{D}(x)\otimes\psi(g).
\end{align*} 
It is easy to see that $D$ is a right \(H\)-comodule BiHom-coalgebra. Moreover, $D$ is an $H$-bicomodule.

Then $D\natural H$ is a L-R-$\binom{m,n,p,q}{s,t,u,v}$-smash biproduct. Its multiplication is defined as follows:
\begin{align*}
\begin{tabular}{c|c c c c}
	                         &$1_{D}\n 1_{H}$ &$1_{D}\n g$ &$x\n 1_{H}$  &$x\n g$\\
	\hline $1_{D}\n 1_{H}$   &$1_{D}\n 1_{H}$ &$1_{D}\n \b(g)$ &$bx\n 1_{H}$ &$bx\n \b(g)$\\
	       $1_{D}\n g$       &$1_{D}\n \a(g)$ &$1_{D}\n 1_{H}$ &$bx\n \a(g)$     &$bx\n 1_{H}$\\
	        $x\n 1_{H} $     &$ax\n 1_{H}$ &$ax\n \b(g)$&0&0\\
           $x\n g $          &$ax\n \a(g)$ &$ax\n 1_{H}$&0&0\\
\end{tabular}
\end{align*}

Its comultiplication and counit are defined as follows:
\begin{align*}
&\D(1_{D}\n 1_{H})=(1_{D}\n 1_{H})\o (1_{D}\n 1_{H}),\quad \v(1_{D}\n 1_{H})=1_{\B},
\\&\D(1_{D}\n g)=(1_{D}\n \om(g))\o (1_{D}\n \psi(g)),\quad \v(1_{D}\n g)=1_{\B},
\\&\D(x\n 1_{H})=(cx\n 1_{H})\o (1_{D}\n \b^{u+1}\om^{v+1}\psi(g))+(1_{D}\n \a^{s+1}\om\psi^{t+1}(g))\o (dx\n 1_{H}),\quad \v(x\n 1_{H})=0,
\\&\D(x\n g)=(cx\n\om(g))\o (1_{D}\n g)+(1_{D}\n 1_{H})\o (dx\n \psi(g)),\quad \v(x\n g)=0.
\end{align*}
\end{example}

\section{BiHom-$\binom{m,n,p,q}{s,t,u,v}$-Yetter-Drinfel'd-Long bimodules category}
\def\theequation{4.\arabic{equation}}
\setcounter{equation} {0}
In this section, we introduce the concept of BiHom-$\binom{m,n,p,q}{s,t,u,v}$-Yetter-Drinfel'd-Long bimodule, $m,n,p,q,s,t,u,v\in \mathbb{Z}$, and prove the category $\mathcal{LR}(H)\binom{m,n,p,q}{s,t,u,v}$ of BiHom-$\binom{m,n,p,q}{s,t,u,v}$-Yetter-Drinfel'd-Long bimodules is a strict braided monoidal category. Moreover, we obtain that $D\n H$ is a BiHom-L-R-$\binom{m,n,p,q}{s,t,u,v}$-smash biproduct bialgebra if and only if $D$ is a BiHom-bialgebra in $\mathcal{LR}(H)\binom{m,n,p,q}{s,t,u,v}$.

Let $H$ be a BiHom-bialgebra. We will describe a strict braided monoidal category associated to $H$, denote by $\mathcal{LR}(H)\binom{m,n,p,q}{s,t,u,v}$. The objects of $\mathcal{LR}(H)\binom{m,n,p,q}{s,t,u,v}$ are vector spaces $M$ endowed with $H$-bimodule and $H$-bicomodule structures (denoted by $h\o m\m h\tr m, m\o h\m m\tl h, m\m m_{[-1]}\o m_{[0]}, m\m m_{(0)}\o m_{(1)}$, for all $m\in M, h\in H$), such that for all $m\in M, h\in H$
\begin{align}
&(\a^{s}\b^{p+1}\om^{q}\psi^{t+1}(h_1)\tr m)_{[-1]}\a\b^{-1}\psi^{-1}\om(h_{2})\o (\a^{s}\b^{p+1}\om^{q}\psi^{t+1}(h_1)\tr m)_{[0]}\notag
\\&\qquad\qquad\quad=h_{1}\a(m_{[-1]})\o \a^{s}\b^{p+1}\om^{q+1}\psi^{t+1}( h_{2}) \tr m_{[0]},\label{e4.1}
\\&(h\tr m)_{(0)}\o(h\tr m)_{(1)}=(\om_{H}(h)\tr m_{(0)})\o \b_{H}(m_{(1)}),\label{e4.2}
\\&[m\tl \a^{m+1}\b^{u}\om^{v+1}\psi^{n}(g_2)]_{(0)}\o\alpha^{-1}\b\psi\omega^{-1}(g_{1})[m\tl \a^{m+1}\b^{u}\om^{v+1}\psi^{n}(g_2)]_{(1)}\notag
\\&\qquad\qquad\quad=m_{(0)}\tl \a^{m+1}\b^{u}\om^{v+1}\psi^{n+1}(g_{1})\o\b(m_{(1)})g_{2},\label{e4.3}
\\&([m\tl g]_{[-1]})\o [m\tl g]_{[0]}=\a_{H}(m_{[-1]})\o m_{[0]}\tl \psi_{H}( g).\label{e4.4}
\end{align}
The morphism in $\mathcal{LR}(H)\binom{m,n,p,q}{s,t,u,v}$ are $H$-bilinear and $H$-bicolinear maps.
\begin{remark}
$(1)$ Eq. $(\ref{e4.1})$ gives the compatible condition for left-left BiHom-$\binom{s,t}{p,q}$-Yetter-Drinfel'd modules, whose category is denoted by ${}^{H}_{H}\mathcal{YD}\binom{s,t}{p,q}$; Eq. $(\ref{e4.3})$ gives the compatible condition for right-right BiHom-$\binom{m,n}{u,v}$-Yetter-Drinfel'd modules , whose category is denoted by $\mathcal{YD}\binom{m,n}{u,v}^{H}_{H}$; 
Eqs. $(\ref{e4.2})$ and $(\ref{e4.4})$ provide the compatible conditions for left-right and right-left BiHom-Long modules , respectively.

$(2)$ If $H$ be a BiHom-Hopf algebra. Eqs. $(\ref{e4.1})$ and $(\ref{e4.3})$ are equivalent to the following forms, respectively.
\begin{align*}
&(\a^{s}\b^{p+1}\om^{q}\psi^{t+1}(h)\tr m)_{[-1]}\o (\a^{s}\b^{p+1}\om^{q}\psi^{t+1}(h)\tr m)_{[0]}
\\&\qquad\qquad =[\a^{-2}\om^{-2}(h_{11})\a^{-1}(m_{[-1]})]\a^{-1}\om^{-1} S(h_{2})\o \a^{s}\b^{p+1}\om^{q-1}\psi^{t+1}( h_{12}) \tr m_{[0]}
\\&(m\tl \a^{m+1}\b^{u}\om^{v+1}\psi^{n}(h))_{(0)}\o (m\tl  \a^{m+1}\b^{u}\om^{v+1}\psi^{n}(h))_{(1)}
\\&\qquad\qquad=m_{(0)}\tl\a^{m+1}\b^{u}\om^{v+1}\psi^{n-1}(h_{21})\o \b^{-1}\psi^{-1} S(h_{1})(\b^{-1}(m_{(1)})\b^{-2}\psi^{-2}(h_{22})),
\end{align*}
see \cite{ZWC24} for detail.
\end{remark}

\begin{example}
Let $H$ be a BiHom-Hopf algebra. Then

$(1)$ $H_{1}=H\o H$ is a BiHom-$\binom{m,n,p,q}{s,t,u,v}$-Yetter-Drinfel'd-Long bimodule with the following structures, for any $h, k, l\in H $:
\begin{align*}
h\tr (k\o l)&=h k\o \b (l),
\\(k\o l)\tl h&=\a (k)\o S\a^{-m-3}\b^{-u+1}\om^{-v-3}\psi^{-n}(h_{1})(\b^{-1}(l) \a^{-m-3}\b^{-u}\om^{-v-3}\psi^{-n}(h_{2})),
\\ \r_{L}(k\o l)&=(k\o l)_{[-1]}\o (k\o l)_{[0]}
\\&=\a^{-s}\b^{-p-3}\om^{-q}\psi^{-t-3}(k_{11})S\a^{-s}\b^{-p-3}\om^{-q+1}\psi^{-t-3}(k_{2})\o (\om^{-1} (k_{12})\o \psi(l)),
\\ \r_{R}(k\o l)&=(k\o l)_{(0)}\o (k\o l)_{(1)}=(\om(k)\o l_{1})\o l_{2}.
\end{align*}

$(2)$ $H_{2}=H\o H$ is a BiHom-$\binom{m,n,p,q}{s,t,u,v}$-Yetter-Drinfel'd-Long bimodule with the following structures, for any $h, k, l\in H $:
\begin{align*}
h\tr (k\o l) &=(\a^{-s}\b^{-p-3}\om^{-q}\psi^{-t-3}(h_{1})\a^{-1}(k))S\a^{-s+1}\b^{-p-3}\om^{-q}\psi^{-t-3}(h_{2})\o \b (l),
\\(k\o l)\tl h&= \a (k)\o lh,
\\ \r_{L}(k\o l)&=(k\o l)_{[-1]}\o (k\o l)_{[0]}=k_{1} \o (k_{2}\o \psi(l)),
\\ \r_{R}(k\o l)&=(k\o l)_{(0)}\o (k\o l)_{(1)}
\\&=(\om(k)\o \psi^{-1}(l_{21}))\o S\a^{-m-3}\b^{-u}\om^{-v-3}\psi^{-n+1}(l_{1})\a^{-m-3}\b^{-u}\om^{-v-3}\psi^{-n}(l_{22}).
\end{align*}
Note that $H\o H$ is also a BiHom-Hopf algebra with usual tensor product and usual tensor coproduct.
\end{example}

One can verify that $\mathcal{LR}(H)\binom{m,n,p,q}{s,t,u,v}$ is a strict monoidal category with the following structures: for all $M, N\in \mathcal{LR}(H)\binom{m,n,p,q}{s,t,u,v}$, and $m\in M,n\in N,h\in H$,
\begin{align*}
&(m\o n)\tl h=m\tl h_{1}\o n\tl h_{2}
\\&\r^{r}(m\o n)=m_{(0)}\o n_{(0)}\o \alpha^{-1}\omega^{-1}(m_{(1)})\b^{-1}\psi^{-1}(n_{(1)})
\\&h\tr (m\o n)=h_{1}\tr m \o h_{2}\tr n
\\&\r^{l}(m\o n)=\alpha^{-1}\omega^{-1}(m_{[-1]})\b^{-1}\psi^{-1}(n_{[-1]})\o  m_{[0]}\o n_{[0]}
\end{align*}
\begin{theorem}
The category $\mathcal{LR}(H)\binom{m,n,p,q}{s,t,u,v}$ is prebraided monoidal category with prebraiding
\begin{align*}
c_{M,N}:M\otimes N\to N\otimes M: m\otimes n\mapsto \a^{s}&\b^{p+1}\om^{q}\psi^{t+1}(m_{[-1]})\tr \b_{N}^{-1}\psi_{N}^{-1}(n_{(0)})
\\&\otimes \a^{-1}_{M}\om_{M}^{-1}(m_{[0]})\tl \a^{m+1}\b^{u}\om^{v+1}\psi^{n}(n_{(1)})
\end{align*}
Moreover, if $H$ is a BiHom-Hopf algebra with a bijective antipode $S$, $\mathcal{LR}(H)\binom{m,n,p,q}{s,t,u,v}$ becomes braided with the inverse of $c_{M,N}$ given by
\begin{align*}
c_{M,N}^{-1}: N\otimes M\to M\otimes N: n\otimes m\mapsto \a_{M} ^{-1}&\om_{M} \psi_{M} ^{-2}(m_{[0]})\tl S^{-1}\a^{m}\b^{u+1}\om^{v}\psi^{n+1}(n_{(1)})
\\&\o S^{-1}\a^{s+1}\b^{p}\om^{q+1}\psi^{t} (m_{[-1]})\tr \b_{N}^{-1}\om_{N} ^{-2}\psi_{N} (n_{(0)})
\end{align*}
for all $M, N\in \mathcal{LR}(H)\binom{m,n,p,q}{s,t,u,v}$, and $m\in M,n\in N,h\in H$.
\end{theorem}
\begin{proof} Let $A=\a^{m}\psi^{n},B=\b^{p}\om^{q},C=\a^{s}\psi^{t},D=\b^{u}\om^{v}$,
first, for all $M, N\in \mathcal{LR}(H)\binom{m,n,p,q}{s,t,u,v}$, and $m\in M,n\in N,h\in H$,
\begin{align*}
&c_{M,N}(h\tr (m\o n))
\\&=c_{M,N}(h_{1}\tr m \o h_{2}\tr n)
\\&=BC\b\psi((h_{1}\tr m )_{[-1]})\tr \b_{N}^{-1}\psi_{N}^{-1}((h_{2}\tr n)_{(0)})
\otimes \a^{-1}_{M}\om_{M}^{-1}((h_{1}\tr m )_{[0]})\tl AD\a\om((h_{2}\tr n)_{(1)})
\\&=BC\b\psi((h_{1}\tr m )_{[-1]})\tr \b_{N}^{-1}\psi_{N}^{-1}[\om(h_{2})\tr n_{(0)}]
\otimes \a^{-1}_{M}\om_{M}^{-1}((h_{1}\tr m )_{[0]})\tl AD\a\b\om(n_{(1)})
\\&=BC\a^{-1}\b\psi[((h_{1}\tr m )_{[-1]})B^{-1}C^{-1}\a\b^{-2}\psi^{-2}\om (h_{2})]\tr \psi_{N}^{-1}(n_{(0)})
\\&\quad\otimes \a^{-1}_{M}\om_{M}^{-1}((h_{1}\tr m )_{[0]})\tl AD\a\b\om(n_{(1)})
\\&=BC\a^{-1}\b\psi[B^{-1}C^{-1}\b ^{-1}\psi ^{-1}(h_{1})\a (m_{[-1]})]\tr \psi_{N}^{-1}(n_{(0)})
\\&\quad\otimes \a^{-1}_{M}\om_{M}^{-1}(\om ( h_{2}) \tr m_{[0]})\tl AD\a\b\om(n_{(1)})
\\&=h_{1}\tr [BC\b\psi(m_{[-1]})\tr \b_{N}^{-1}\psi_{N}^{-1}(n_{(0)})]
\otimes h_{2} \tr [\a^{-1}_{M}\om_{M}^{-1}m_{[0]}\tl AD\a\om(n_{(1)})]
\\&=h\tr c_{M,N}(m\o n),
\end{align*} 
thus $c_{M,N}$ is left $H$-linear. Similarly, one can verify that $c_{M,N}$ is right $H$-linear.
\begin{align*}
&c_{M,N}(m\o n)_{[-1]}\o c_{M,N}(m\o n)_{[0]}
\\&=\alpha ^{-1}\omega ^{-1}([BC\b \psi (m_{[-1]})\tr \b_{N} ^{-1}\psi_{N} ^{-1}(n_{(0)})]_{[-1]})
\b ^{-1}\psi ^{-1}([\a_{M} ^{-1}\om _{M}^{-1}(m_{[0]})\tl AD\a \om (n_{(1)})]_{[-1]})
\\&\quad\o  [BC\b \psi (m_{[-1]})\tr \b_{N} ^{-1}\psi_{N} ^{-1}(n_{(0)})]_{[0]}
\o [\a _{M}^{-1}\om_{M} ^{-1}(m_{[0]})\tl AD\a \om (n_{(1)})]_{[0]}
\\&=\alpha ^{-1}\omega ^{-1}([BC\b \psi (m_{[-1]})\tr \b_{N} ^{-1}\psi_{N} ^{-1}(n_{(0)})]_{[-1]})
\b ^{-1}\psi ^{-1}\om ^{-1}(m_{[0][-1]})
\\&\quad\o  [BC\b \psi (m_{[-1]})\tr \b_{N} ^{-1}\psi_{N} ^{-1}(n_{(0)})]_{[0]}
\o \a_{M} ^{-1}\om_{M} ^{-1}(m_{[0][0]})\tl AD\a \om \psi(n_{(1)})
\\&=\alpha ^{-1}\omega ^{-1}([BC\b \om^{-1}\psi (m_{[-1]1})\tr \b _{N}^{-1}\psi_{N} ^{-1}(n_{(0)})]_{[-1]}\a\b ^{-1}\psi ^{-1}(m_{[-1]2}))
\\&\quad\o  [BC\b \psi (m_{[-1]})\tr \b_{N} ^{-1}\psi_{N} ^{-1}(n_{(0)})]_{[0]}
\o \a_{M} ^{-1}\om_{M} ^{-1}\psi_{M} (m_{[0]})\tl AD\a \om \psi(n_{(1)})
\\&=\alpha ^{-1}\omega ^{-1}(\om^{-1}(m_{[-1]1})\a \b ^{-1}\psi ^{-1}(n_{(0)[-1]}))
\\&\quad\o  BC\b \psi  (m_{[-1]2}) \tr \b_{N} ^{-1}\psi_{N} ^{-1}(n_{(0)[0]})
\o \a_{M} ^{-1}\om_{M} ^{-1}\psi_{M}(m_{[0]})\tl AD\a \om \psi(n_{(1)})
\\&=\alpha ^{-1}\omega ^{-1}(m_{[-1]})\b ^{-1}\psi ^{-1}(n_{[-1]})
\\&\quad\o BC\b \psi (m_{[0][-1]})\tr \b_{N} ^{-1}\psi_{N} ^{-1}(n_{[0](0)})
\otimes \a_{M} ^{-1}\om_{M} ^{-1}(m_{[0][0]})\tl AD\a \om (n_{[0](1)})
\\&=\alpha ^{-1}\omega ^{-1}(m_{[-1]})\b ^{-1}\psi ^{-1}(n_{[-1]})\o c_{M,N}[m_{[0]}\o n_{[0]}]
\\&=(m\o n)_{[-1]}\o c_{M,N}[(m\o n)_{[0]}]
\end{align*} 
thus $c_{M,N}$ is left $H$-colinear. Similarly, one can verify that $c_{M,N}$ is right $H$-colinear.

It is straightforward to check that $c_{M,N}$ satisfies the hexagonal equations. Hence, it is a prebraiding in $\mathcal{LR}(H)\binom{m,n,p,q}{s,t,u,v}$. Furthermore, if $H$ has bijective antipode $S$,
\begin{align*}
&c_{M,N}^{-1}\circ c_{M,N}(m\o n)
\\&=\a_{M}^{-1}\om_{M}\psi_{M} ^{-2}([\a_{M} ^{-1}\om_{M}^{-1}(m_{[0]})\tl AD\a \om (n_{(1)})]_{[0]})
\\&\quad\tl S^{-1}AD\b\psi([BC\b \psi (m_{[-1]})\tr \b_{N} ^{-1}\psi_{N} ^{-1}(n_{(0)})]_{(1)})
\\&\quad \o S^{-1}BC\a\om([\a _{M}^{-1}\om_{M} ^{-1}(m_{[0]})\tl AD\a \om (n_{(1)})]_{[-1]})
\\&\quad\tr \b_{N}^{-1}\om_{N} ^{-2}\psi_{N}([BC\b \psi (m_{[-1]})\tr \b_{N} ^{-1}\psi_{N} ^{-1}(n_{(0)})]_{(0)})
\\&=\a_{M}^{-1}\om_{M}\psi_{M} ^{-2}[\a_{M} ^{-1}\om_{M} ^{-1}(m_{[0][0]})\tl AD\a \psi\om (n_{(1)})]
\tl S^{-1}AD\b(n_{(0)(1)})
\\&\quad \o S^{-1}BC\a(m_{[0][-1]})
\tr \b_{N}^{-1}\om_{N} ^{-2}\psi_{N}[BC\b \psi \om(m_{[-1]})\tr \b_{N} ^{-1}\psi_{N} ^{-1}(n_{(0)(0)})]
\\&=\a_{M}^{-1}\psi_{M} ^{-2}(m_{[0][0]})
\tl [AD\om^{2} \psi^{-1}(n_{(1)}) S^{-1}AD(n_{(0)(1)})]
\\&\quad \o [S^{-1}BC(m_{[0][-1]})BC\psi^{2} \om^{-1}(m_{[-1]})]
\tr\b_{N}^{-1}\om_{N} ^{-2}(n_{(0)(0)})
\\&=\a_{M}^{-1}\psi_{M} ^{-1}(m_{[0]})
\tl [AD\om^{2}\psi^{-2}(n_{(1)2}) S^{-1}AD(n_{(1)1})]
\\&\quad \o [S^{-1}BC(m_{[-1]2})BC\psi^{2}\om ^{-2}(m_{[-1]1})]
\tr\b_{N}^{-1}\om_{N} ^{-1}(n_{(0)})
\\&=m\o n,
\end{align*} 
thus $c_{M,N}^{-1}\circ c_{M,N}=id_{U\o V}$. Similarly, we have $c_{M,N}\circ c_{M,N}^{-1}=id_{V\o U}$. Then $c_{M,N}$ is invertible, thus it is a braiding.

The proof is completed.
\end{proof}
\begin{definition}
Let $H$ be a BiHom-bialgebra, $D$ a vector space and $\a_{D},\b_{D},\om_{D},\psi_{D}\in Aut(D)$. We say that $D$ is a BiHom-bialgebra in $\mathcal{LR}(H)\binom{m,n,p,q}{s,t,u,v}$ if the following axioms hold:
\begin{enumerate}
\item[$(i)$] $D$ is a Yetter-Drinfel'd-Long bimodule.
\item[$(ii)$] $(D,\mu_{D},1_{D},\a_{D},\b_{D})$ is a BiHom-algebra in $\mathcal{LR}(H)\binom{m,n,p,q}{s,t,u,v}$, i.e., $D$ is a $H$-bimodule BiHom-algebra and $D$ is a $H$-bicomodule BiHom-algebra.
\item[$(iii)$] $(D,\Delta_{D},\varepsilon_{D},\om_{D},\psi_{D})$ is a BiHom-coalgebra in $\mathcal{LR}(H)\binom{m,n,p,q}{s,t,u,v}$, i.e., $D$ is a $H$-bimodule BiHom-coalgebra and $D$ is a $H$-bicomodule BiHom-coalgebra.
\item[$(iv)$] We have the following identities:
\begin{enumerate}
\item[$(iv$-$1)$] $\v_{D}(1_{D})=1, \v_{D}(ab)=\v_{D}(a)\v_{D}(b),$
\item[$(iv$-$2)$] $\D_{D}(1_{D})=1_{D}\o 1_{D},$
\item[$(iv$-$3)$] $\Delta_D(ab) = a_{1}(\a^{s+1}\b^{p}\om^{p+1}\psi^{t}(a_{2[-1]})\tr\beta_D^{-1}\om_D^{-1}(b_{1(0)}))\notag 
\\  \o (\alpha_D^{-1}\psi_D^{-1}(a_{2[0]})\tl \a^{m}\b^{u+1}\om^{v}\psi^{n+1}(b_{1(1)}))b_{2},
$
\item[$(iv$-$3)$] $\a_{D},\b_{D}$ are BiHom-coalgebra maps, $\om_{D},\psi_{D}$ are BiHom-algebra maps, and they commute with each other.
\end{enumerate}
\end{enumerate}
\end{definition}
\begin{proposition}
Let $H$ be a BiHom-bialgebra, $D$ a vector space and $\a_{D},\b_{D},\om_{D},\psi_{D}\in Aut(D)$. Then $(H,D)$ is an L-R-$\binom{m,n,p,q}{s,t,u,v}$-admissible pair if and only if $D$ is a BiHom-bialgebra in $\mathcal{LR}(H)\binom{m,n,p,q}{s,t,u,v}$ satisfying the condition $(\ref{e3.12})$.
\end{proposition}
\begin{proof}
A straightforward verification.
\end{proof}
\section{Categorical Equivalence}
\def\theequation{5.\arabic{equation}}
\setcounter{equation} {0}
In this section, let $H$ be a finite dimensional BiHom-bialgebra, we prove that $\mathcal{LR}(H)\binom{m,n,p,q}{s,t,u,v}$ is isomorphic to the category $\!^{H\o H^{*}}_{H\o H^{*}}\mathcal{YD}\binom{s,t}{p,q}$ of left left BiHom-$\binom{s,t}{p,q}$-Yetter-Drinfel'd modules over BiHom-bialgebra $H\o H^{*}=(H\o H^{*}, \mu\o \star, 1\o \varepsilon, \D\o \Delta_{H^*}, \v\o \varepsilon_{H^*}, \a\o \alpha_{H^*}, \b\o \beta_{H^*}, \om\o \om_{H^*}, \psi\o \psi_{H^*})$ as strict braided monoidal categories.
\begin{lemma}
Let $H$ be a finite dimensional BiHom-bialgebra. Then we have a functor $F:\mathcal{LR}(H)\binom{m,n,p,q}{s,t,u,v}\ra \!^{H\o H^{*}}_{H\o H^{*}}\mathcal{YD}\binom{s,t}{p,q}$ given for any object $M\in \mathcal{LR}(H)\binom{m,n,p,q}{s,t,u,v}$ and any morphism $\vartheta$ by
$$F(M)=M,\quad  F(\vartheta)=\vartheta,$$
where $H\o H^{*}$ is a BiHom-bialgebra with tensor product and tensor coproduct.
\end{lemma}
\begin{proof}
For all $M\in \mathcal{LR}(H)\binom{m,n,p,q}{s,t,u,v}$, let $A=\a^{m}\psi^{n},B=\b^{p}\om^{q},C=\a^{s}\psi^{t},D=\b^{u}\om^{v},E=\a^{m+s+2}\b^{p+u+2}\om^{q+v+2}\psi^{n+t+2}$, first of all, define the left action of $H\o H^{*}$ on $M$ by
\begin{align*}
(h\o f)\c m=f(m_{(1)})E^{-1}(h)\tr \om_{M} ^{-1}(m_{(0)}),
\end{align*}
for all $h\in H, f\in H^{*}$ and $m\in M$. Then $M$ is a left $H\o H^{*}$-module. Indeed,
\begin{align*}
&(h\o f)(h'\o f')\c \b_{M}(m)
\\&=f f'(\b (m_{(1)}))E^{-1}(hh')\tr \om_{M} ^{-1}(\b_{M} (m_{(0)}))
\\&=f \alpha^{-1}\b\om^{-1}(m_{(1)1})f'\psi^{-1}(m_{(1)2})E^{-1}\a(h)\tr [ E^{-1}(h')\tr\om_{M} ^{-1}(m_{(0)})]
\\&=f'(m_{(1)})f\a^{-1}\b\om ^{-1}(m_{(0)(1)})E^{-1}\a(h)\tr \om_{M} ^{-1}[E^{-1}\om (h')\tr \om_{M} ^{-1}(m_{(0)(0)})]
\\&=f'(m_{(1)})f\a^{-1}([E^{-1}(h')\tr \om_{M} ^{-1}(m_{(0)})]_{(1)})E^{-1}\a(h)\tr \om_{M} ^{-1}([E^{-1}(h')\tr \om_{M} ^{-1}(m_{(0)})]_{(0)})
\\&=f'(m_{(1)})(\a (h)\o \a_{H^{*}}(f))\c [E^{-1}(h')\tr \om_{M} ^{-1}(m_{(0)})]
\\&=(\a (h)\o \a_{H^{*}}(f))\c [(h'\o f')\c m]
\end{align*} 
And $(1\o \v)\c m=\<\v, m_{(1)}\>1\tr \om_{M} ^{-1}(m_{(0)})=\b_{M} (m)$. 

Then, $(\a\o \a_{H^*})(h\o f)\c \a_M(m)=\a_M((h\o f)\c m),(\b\o \b_{H^*})(h\o f)\c \b_M(m)=\b_M((h\o f)\c m),(\om\o \om_{H^*})(h\o f)\c \om_M(m)=\om_M((h\o f)\c m),(\psi\o \psi_{H^*})(h\o f)\c \psi_M(m)=\psi_M((h\o f)\c m),$
$M$ is a left $H\o H^{*}$-module. Next, define the left coaction of $H\o H^{*}$ on $M$ by
\begin{align*}
\r(m)=m_{<-1>}\o m_{<0>}=E&(m_{[-1]})\o e^{i}\o \a_{M}^{-1}(m_{[0]})\tl E\a^{-1}\om^{-1}(e_{i}),
\end{align*}
where ${e_{i}}$ and ${e^{i}}$ are dual bases of $H$. We then prove that $M$ is a left $H\o H^{*}$-comodule. On one hand,
\begin{align*}
&(\om \o \om_{H^*}\o \r)\circ \r(m)
\\&=E\om (m_{[-1]})\o \om_{H^*}(e^{i})\o E([\a_{M}^{-1}(m_{[0]})\tl E\a^{-1}\om^{-1}(e_{i})]_{[-1]})\o e^{j}
\\&\quad\o \a_{M}^{-1}([\a_{M}^{-1}(m_{[0]})\tl E\a^{-1}\om^{-1}(e_{i})]_{[0]})\tl E\a^{-1}\om^{-1}(e_{j})
\\&=E\om (m_{[-1]})\o e^{i}
\o E(m_{[0][-1]})\o e^{j}\o \a_{M}^{-1}[\a_{M}^{-1}(m_{[0][0]})\tl E\a^{-1}\om^{-2}\psi(e_{i})]\tl E\a^{-1}\om^{-1}(e_{j})
\\&=E\om (m_{[-1]})\o e^{i}
\o E(m_{[0][-1]})\o e^{j}\o \a_{M}^{-1}(m_{[0][0]})\tl E\a^{-2}\om^{-2}\psi(e_{i})E\a^{-1}\b^{-1}\om^{-1}(e_{j})
\\&=E(m_{[-1]1})\o e^{i}
\o E(m_{[-1]2})\o e^{j}\o \a_{M}^{-1}\psi_{M}(m_{[0]})\tl E\a^{-2}\om^{-2}\psi(e_{i}) E\a^{-1}\b^{-1}\om^{-1}(e_{j}).
\end{align*} 
Evaluating the right side of equation on $id\o h\o id \o g\o id$, we obtain
\begin{align*}
E(m_{[-1]1})\o E(m_{[-1]2})\o  \a_{M}^{-1}\psi_{M}(m_{[0]})\tl E\a^{-2}\om^{-2}\psi(h) E\a^{-1}\b^{-1}\om^{-1}(g).
\end{align*}
On the other hand,
\begin{align*}
&(\Delta_{H\otimes H^*} \otimes \psi_{M})\rho(m)
\\&=E(m_{[-1]1})\o e^{i}_{1}\o E(m_{[-1]2})\o e^{i}_{2}\o \a_{M}^{-1}\psi_{M}(m_{[0]})\tl E\a^{-1}\om^{-1}\psi(e_{i} ).
\end{align*}
Evaluating the right side of equation on $id\o h\o id \o g\o id$, we obtain
\begin{align*}
E(m_{[-1]1})\o E(m_{[-1]2})\o  \a_{M}^{-1}\psi_{M}(m_{[0]})\tl E\a^{-2}\om^{-2}\psi(h) E\a^{-1}\b^{-1}\om^{-1}(g).
\end{align*}
Since $h, g\in H$ are arbitrary, we have
\begin{align*}
(\om \o \om_{H^*}\o \r)\circ \r(m)=(\Delta_{H\otimes H^*} \otimes \psi)\rho(m).
\end{align*}
Since $(\v\o \v^{*}\o \mathrm{id}_{M})\r(m)=\v(m_{[-1]})\a_{M}^{-1}(m_{[0]})\tl 1_{H}=\psi_{M}(m)$, the other identity is straightforward and left to the reader.

Finally,
\begin{align*}
&([BC\b \psi (h_{1})\o B_{H^{*}}C_{H^{*}}\b_{H^*}\psi_{H^*}(f_{1})]\c m)_{<-1>}
\\&\quad[\a \b ^{-1}\psi ^{-1}\om (h_{2})\o \a_{H^*}\b_{H^*}^{-1}\psi_{H^*}^{-1}\om_{H^*}(f_{2})]
 \o ([BC\b \psi (h_{1})\o B_{H^{*}}C_{H^{*}}\b_{H^*}\psi_{H^*}(f_{1})]\c m)_{<0>}
\\&=f_{1}B^{-1}C^{-1}\b ^{-1}\psi ^{-1}(m_{(1)})(E^{-1}BC\b \psi (h_{1})\tr \om_{M} ^{-1}(m_{(0)}))_{<-1>}
\\&\quad[\a \b ^{-1}\psi ^{-1}\om (h_{2})\o \a_{H^*}\b_{H^*}^{-1}\psi_{H^*}^{-1}\om_{H^*}(f_{2})]
 \o (E^{-1}BC\b \psi (h_{1})\tr \om_{M} ^{-1}(m_{(0)}))_{<0>}
\\&=f_{1}B^{-1}C^{-1}\b ^{-1}\psi ^{-1}(m_{(1)})
E([E^{-1}BC\b \psi (h_{1})\tr \om_{M} ^{-1}(m_{(0)})]_{[-1]}E^{-1}\a \b ^{-1}\psi ^{-1}\om (h_{2}))
\\&\quad\o e^{i}\star\a_{H^*}\b_{H^*}^{-1}\psi_{H^*}^{-1}\om_{H^*}(f_{2})
 \o \a_{M}^{-1}([E^{-1}BC\b \psi (h_{1})\tr \om_{M} ^{-1}(m_{(0)})]_{[0]})\tl E\a^{-1}\om^{-1}(e_{i})
\\&=f_{1}B^{-1}C^{-1}\b ^{-1}\psi ^{-1}(m_{(1)})
h_{1}E\a \om ^{-1}(m_{(0)[-1]})
\\&\quad\o e^{i}\star\a_{H^*}\b_{H^*}^{-1}\psi_{H^*}^{-1}\om_{H^*}(f_{2})
 \o \a_{M}^{-1}(E^{-1}BC\b \psi \om  ( h_{2}) \tr \om_{M} ^{-1}(m_{(0)[0]})\tl E\a^{-1}\om^{-1}(e_{i}).
\end{align*}
Evaluating the right side of the equation on $id\o g\o id$, we obtain
\begin{align*}
&f[B^{-1}C^{-1}\alpha ^{-1}\b ^{-1}\psi ^{-1}\om ^{-1}(m_{(1)})\a ^{-1}\beta ^{-1}\psi ^{-1}\om ^{-1}(g_{2})]h_{1}E\a \om^{-1}(m_{(0)[-1]})
\\&\quad \o \a_{M}^{-1}[E^{-1}BC\b \psi \om  ( h_{2}) \tr \om_{M} ^{-1}(m_{(0)[0]})]\tl E\a^{-2}\om^{-2}(g_{1})
\end{align*}
And
\begin{align*}
&(h_{1}\o f_{1})(\a \o \a_{H^*})(m_{<-1>})\o [BC\b \psi \om  (h_{2})\o B_{H^{*}}C_{H^{*}}\b_{H^*}\psi_{H^*}\om_{H^*} (f_{2})] \c m_{<0>}
\\&=h_{1}E\a (m_{[-1]})\o f_{1}\star \a_{H^*}(e^{i})
\\&\quad\o [BC\b \psi \om  (h_{2})\o B_{H^{*}}C_{H^{*}}\b_{H^*}\psi_{H^*}\om_{H^*} (f_{2})] \c [\a_{M}^{-1}(m_{[0]})\tl E\a^{-1}\om^{-1}(e_{i})]
\\&=f_{2}B^{-1}C^{-1}\b ^{-1}\psi ^{-1}\om ^{-1}([\a_{M}^{-1}(m_{[0]})\tl E\a^{-2}\om^{-1}(e_{i})]_{(1)})
\\&\quad h_{1}E\a (m_{[-1]})\o f_{1}\star e^{i}\o E^{-1}BC\b \psi \om  (h_{2})\tr \om_{M} ^{-1}([\a_{M}^{-1}(m_{[0]})\tl E\a^{-2}\om^{-1}(e_{i})]_{(0)}).
\end{align*}
Evaluating the right side of the equation on $id\o g\o id$, we obtain
\begin{align*}
&f[\alpha ^{-2}\om ^{-2}(g_{1})B^{-1}C^{-1}\b ^{-2}\psi ^{-2}\om ^{-1}([\a_{M}^{-1}(m_{[0]})\tl 
E\a^{-2}\beta ^{-1}\om^{-1}\psi ^{-1}(g_{2})]_{(1)})]
\\&\quad h_{1}E\a (m_{[-1]})\o E^{-1}BC\b \psi \om  (h_{2})\tr \om_{M} ^{-1}([\a_{M}^{-1}(m_{[0]})\tl E\a^{-2}\beta ^{-1}\om^{-1}\psi ^{-1}(g_{2})]_{(0)})
\\&=fB^{-1}C^{-1}\b ^{-2}\psi ^{-2}\om ^{-1}[BC\alpha ^{-2}\b ^{2}\psi ^{2}\om ^{-1}(g_{1})[\a_{M}^{-1}(m_{[0]})\tl
 E\a^{-2}\beta ^{-1}\om^{-1}\psi ^{-1}(g_{2})]_{(1)}]
\\&\quad h_{1}E\a (m_{[-1]})\o E^{-1}BC\b \psi \om  (h_{2})\tr \om_{M} ^{-1}([\a_{M}^{-1}(m_{[0]})\tl E\a^{-2}\beta ^{-1}\om^{-1}\psi ^{-1}(g_{2})]_{(0)})
\\&=fB^{-1}C^{-1}\b ^{-2}\psi ^{-2}\om ^{-1}[\a^{-1}\b (m_{[0](1)})BC\alpha ^{-1}\b \psi (g_{2})]
\\&\quad h_{1}E\a (m_{[-1]})\o E^{-1}BC\b \psi \om  (h_{2})\tr \om_{M} ^{-1}(\a_{M}^{-1}(m_{[0](0)})\tl ABCD\b \psi ^{2}\om (g_{1}))
\\&=f[B^{-1}C^{-1}\alpha ^{-1}\b ^{-1}\psi ^{-2}\om ^{-1}(m_{[0](1)})
\a ^{-1}\beta ^{-1}\psi ^{-1}\om ^{-1}(g_{2})]
\\&\quad h_{1}E\a (m_{[-1]})\o E^{-1}BC \b \psi \om  ( h_{2}) \tr [\a_{M}^{-1}\om_{M} ^{-1}(m_{[0](0)})\tl E\alpha ^{-2}\b ^{-1}\om ^{-2}(g_{1})]
\\&=f[B^{-1}C^{-1}\alpha ^{-1}\b ^{-1}\psi ^{-1}\om ^{-1}(m_{(1)})
\a ^{-1}\beta ^{-1}\psi ^{-1}\om ^{-1}(g_{2})]
\\&\quad h_{1}E\a\om^{-1} (m_{(0)[-1]})\o [E^{-1}BC \a^{-1}\b \psi \om  ( h_{2}) \tr \a_{M}^{-1}\om_{M} ^{-1}(m_{(0)[0]})]\tl E\alpha ^{-2}\om ^{-2}(g_{1}).
\end{align*}
Therefore $M$ is a left left BiHom-$\binom{s,t}{p,q}$-Yetter-Drinfel'd modules over $H\o H^{*}$. It is straightforward to verify that any morphism in $\mathcal{LR}(m,n,p,q)$ is also a morphism in $\!^{H\o H^{*}}_{H\o H^{*}}\mathcal{YD}\binom{s,t}{p,q}$.

The proof is completed.
\end{proof}
\begin{lemma}
Let $H$ be a BiHom-bialgebra. Then we have a functor $G:\!^{H\o H^{*}}_{H\o H^{*}}\mathcal{YD}\binom{s,t}{p,q}\ra \mathcal{LR}(H)\binom{m,n,p,q}{s,t,u,v}$ given for any object $M\in \!^{H\o H^{*}}_{H\o H^{*}}\mathcal{YD}\binom{s,t}{p,q}$ and any morphism $\theta$ by
$$G(M)=M,\quad G(\theta)=\theta.$$
\end{lemma}
\begin{proof}
Let $A=\a^{m}\psi^{n},B=\b^{p}\om^{q},C=\a^{s}\psi^{t},D=\b^{u}\om^{v},E=\a^{m+s+2}\b^{p+u+2}\om^{q+v+2}\psi^{n+t+2}$, for any $M\in \!^{H\o H^{*}}_{H\o H^{*}}\mathcal{YD}\binom{s,t}{p,q}$, denote the left $H\o H^{*}$-coaction on $M$ by
\begin{align*}
m\m m_{<-1>}\o m_{<0>},
\end{align*}
for all $m\in M$. Define the $H$-bimodule and $H$-bicomodule structures as follows:
\begin{align*}
&h \tr m = (E(h) \otimes \v) \c m,
\\&\r_L(m) = m_{[-1]} \otimes m_{[0]} =(E^{-1}\otimes \v^*)(m_{<-1>}) \otimes m_{<0>},
\\&m \tl h =\big<(\v\otimes \mathrm{id})m_{<-1>},\, E^{-1}\om\psi(h) \big> \a_M\psi_M^{-1}(m_{<0>}),
\\&\r_R(m) = m_{(0)} \otimes m_{(1)} =(1 \otimes e^i) \c \beta_{M}^{-1}\om_{M}(m) \otimes \b\om^{-1}(e_i),
\end{align*}
for all $m\in M,h\in H$.

Obviously, $M$ is a left $H$-module. And
\begin{align*}
&(\D\o \psi_{M})\circ \r_{L}(m)
\\&=(E^{-1} \otimes \v^*)(m_{<-1>1}) \o (E^{-1}\otimes \v^*)(m_{<-1>2}) \o \psi_{M}(m_{<0>})
\\&=(E^{-1}\om  \otimes \v^*)(m_{<-1>}) \otimes (E^{-1}\otimes \v^*)(m_{<0><-1>}) \otimes m_{<0><0>}
\\&=(\om \o\om_{H^{*}}\o \r_{L})\circ \r_{L}(m).
\end{align*} 
The other identities is straightforward. Thus, $M$ is a left $H$-comodule. For all $h, g\in H, m\in M$,
\begin{align*}
&[m \tl h]\tl \b (g)
\\&=[\big<(\v\otimes \mathrm{id})m_{<-1>},\, E^{-1}(h) \big> \a_{M}\psi_{M}^{-1}(m_{<0>})]\tl \b(g)
\\&=\big<(\v\otimes \mathrm{id})m_{<-1>},\, E^{-1}\om\psi(h) \big> \big<(\v\otimes \mathrm{id})m_{<0><-1>},\, E^{-1}\a^{-1}\b\om\psi^{2} (g) \big> \a_{M}^{2}\psi_{M}^{-2}(m_{<0><0>})
\\&=\big<(\v\otimes \mathrm{id})m_{<-1>1},\, E^{-1}\om^{2}\psi (h) \big> \big<(\v\otimes \mathrm{id})m_{<-1>2},\, E^{-1}\a^{-1}\b\om\psi^{2} (g) \big> \a_{M}^{2}\psi_{M}^{-1}(m_{<0>})
\\&=\big<(\v\otimes \mathrm{id})m_{<-1>},\, E^{-1}\a ^{-1}\om\psi(hg) \big> \a_{M}^{2}\psi_{M}^{-1}(m_{<0>})
\\&=\a_{M}(m) \tl hg.
\end{align*} 
The other identities is obvious. Thus $M$ is a right $H$-module. Since
\begin{align*}
&(\om_{M}\o \D)\circ \r_{R}(m)
\\&=(1 \otimes e^i) \c \b_{M}^{-1}\om_{M}^{2}(m) \otimes \b\om^{-2}(e_{i1})\otimes \b\om^{-2}(e_{i2})
\\&=(1 \otimes e^i\star e^j) \c \b_{M}^{-1}\om_{M}^{2}(m) \otimes \a\b\om^{-1} (e_{i})\otimes \b^{2}\om^{-2}\psi (e_{j})
\\&=(1 \otimes \a_{H^{*}}^{-1}(e^i))((1 \otimes \b_{H^{*}}^{-1}\om_{H^{*}}(e^j)) \c  \b_{M}^{-1}\om_{M}^{2}(m)) \otimes \b\om^{-1}(e_i) \otimes \b\om^{-1}\psi (e_j)
\\&=(1 \otimes e^i) \c \b_{M}^{-1}\om_{M}((1 \otimes e^j) \c \b_{M}^{-1}\om_{M}(m)) \otimes \b\om^{-1}(e_i) \otimes \b\om^{-1}\psi (e_j)
\\&=(\r_{R}\o \psi )\circ \r_{R}(m).
\end{align*} 
Thus $M$ is a right $H$-comodule. Moreover
\begin{align*}
&\a (h) \tr (m\tl g)
\\&=\big<(\v\otimes \mathrm{id})m_{<-1>},\, E^{-1}\om\psi(g) \big>(E\a (h) \otimes \v) \c  \a_{M}\psi_{M}^{-1}(m_{<0>})
\\&=\big<(\v\otimes \mathrm{id})(\b \o \b_{H^{*}})m_{<-1>}
,\,  E^{-1}\b\om\psi (g) \big> \a_{M}\psi_{M}^{-1}((E\psi ( h) \otimes \v) \c m_{<0>})
\\&=\big<(\v\otimes \mathrm{id})(\alpha ^{-1}\o \alpha_{H^{*}}^{-1})
[(B^{-1}C^{-1}E\b ^{-1}\psi ^{-1}\om ^{-1}(h_{1})\o \v)(\alpha \o \alpha_{H^{*}})(m_{<-1>})]
,\, \\&\quad E^{-1}\b\om\psi (g) \big> \a_{M}\psi_{M}^{-1}((E( h_{2}) \otimes \v) \c m_{<0>})
\\&=\big<(\v\otimes \mathrm{id})(\alpha ^{-1}\o \alpha_{H^{*}}^{-1})[((E\om ^{-1}(h_{1}) \otimes \v) \c m)_{<-1>}
(B^{-1}C^{-1}E\a\b ^{-2}\psi ^{-2}(h_{2}) \otimes \v)]
,\, \\&\quad E^{-1}\b\om\psi (g) \big> \a_{M}\psi_{M}^{-1}(((E\om ^{-1}(h_{1}) \otimes \v) \c m)_{<0>})
\\&=\big<(\v\otimes \mathrm{id})((E(h) \otimes \v) \c m)_{<-1>},\, E^{-1}\b\om\psi (g) \big> \a_{M}\psi_{M}^{-1}(((E(h) \otimes \v) \c m)_{<0>})
\\&=(h \tr m)\tl \b (g).
\end{align*}
Thus $M$ is an BiHom-$H$-bimodule. And
\begin{align*}
&(\om \o \r_{R})\circ \r_{L}(m)
\\&=(E^{-1}\om  \otimes \v^*)(m_{<-1>}) \otimes (1 \otimes e^i) \c \b_{M}^{-1}\om_{M}(m_{<0>}) \otimes \b\om^{-1}(e_i)
\\&=(E^{-1} \otimes \v^*)(\alpha ^{-1}\o \alpha_{H^{*}}^{-1})[(1\o B^{-1}C^{-1}\b ^{-1}\psi ^{-1}\om ^{-1}(e^i_{1}))(\a \b_{M}^{-1}\om_{M}\o \a_{H^{*}}\b_{H^{*}}^{-1}\om_{H^{*}})(m_{<-1>})]
\\&\quad\otimes (1\o e^i_{2}) \c\b_{M}^{-1}\om_{M}(m_{<0>}) \otimes \b\om^{-1}\psi (e_i)
\\&=(E^{-1} \otimes \v^*)(\alpha ^{-1}\o \alpha_{H^{*}}^{-1})[((1 \otimes \om_{H^{*}}^{-1}(e^i_{1})) \c \b_{M}^{-1}\om_{M}(m))_{<-1>}
(1 \otimes B^{-1}C^{-1}\a\b ^{-2}\psi ^{-2}(e^i_{2}))] 
\\&\quad\otimes ((1 \otimes \om_{H^{*}}^{-1}(e^i_{1})) \c \b_{M}^{-1}\om_{M}(m))_{<0>} \otimes \b\om^{-1}\psi (e_i)
\\&=(E^{-1} \otimes \v^*)(((1 \otimes e^i) \c \b_{M}^{-1}\om_{M}(m))_{<-1>}) \otimes ((1 \otimes e^i) \c \b_{M}^{-1}\om_{M}(m))_{<0>} \otimes \b\om^{-1}\psi (e_i)
\\&=(\r_{L}\o \psi )\circ \r_{R}(m).
\end{align*} 
Thus $M$ is an $H$-bicomodule.

For all $h\in H, m\in M$, we now prove $(\ref{e4.3})$: On one hand,
\begin{align*}
&[m\tl AD\a \om (h_2)]_{(0)}\o\alpha ^{-1}\b \psi \omega ^{-1}(h_{1})[m\tl AD\a \om (h_2)]_{(1)}
\\&=\big<(\v\otimes \mathrm{id})m_{<-1>},\, E^{-1} AD\a \om^{2}\psi(h_2) \big>\a_{M}\psi_{M}^{-1}(m_{<0>})_{(0)}\o\alpha ^{-1}\b \psi \omega ^{-1}(h_{1})\a_{M}\psi_{M}^{-1}(m_{<0>})_{(1)}
\\&=\big<(\v\otimes \mathrm{id})m_{<-1>},\, E^{-1}AD\a \om^{2}\psi (h_2) \big>(1 \otimes e^i) \c \a_{M}\b_{M}^{-1}\om_{M}\psi_{M}^{-1}(m_{<0>})\o\alpha ^{-1}\b \psi \omega ^{-1}(h_{1})\b\om^{-1}(e_i).
\end{align*} 
Evaluating the right side on $id\o f$, for all $f\in H^{*}$, we have
\begin{align*}
\big<(\v\otimes \mathrm{id})m_{<-1>},\, E^{-1}AD\a \om^{2}\psi (h_2) \big>f_{1}\b \psi (h_{1})(1 \otimes \b_{H^*}^{-2}\om_{H^*}\psi_{H^*}^{-1}(f_{2})) \c \a_{M}\b_{M}^{-1}\om_{M}\psi_{M}^{-1}(m_{<0>}).
\end{align*}
On the other hand,
\begin{align*}
&m_{(0)}\tl AD\a \om\psi  (h_{1})\o\b (m_{(1)})h_{2}
\\&=[(1 \otimes e^i) \c \b_{M}^{-1}\om_{M}(m)]\tl AD\a \om\psi (h_{1})\o \b^{2}\om^{-1}(e_i)h_{2}
\\&=\big<(\v\otimes \mathrm{id})[(1 \otimes e^i) \c \b_{M}^{-1}\om_{M}(m)]_{<-1>},\, E^{-1}AD\a\om^{2} \psi^{2}  (h_{1}) \big>
 \\&\quad\a_{M}\psi_{M}^{-1}([(1 \otimes e^i) \c \b_{M}^{-1}\om_{M}(m)]_{<0>})\o \b^{2}\om^{-1}(e_i)h_{2}.
\end{align*}
Evaluating the right side on $id\o f$, for all $f\in H^{*}$, we have
\begin{align*}
&\big<(\v\otimes \mathrm{id})[(1 \otimes \a_{H^{*}}^{-1}\b_{H^{*}}^{-2}(f_{1})) \c \b_{M}^{-1}\om_{M}(m)]_{<-1>},\, E^{-1}AD\a\om^{2} \psi^{2} (h_{1}) \big>f_{2}\b \psi (h_{2})
\\&\quad \a_{M}\psi_{M}^{-1}([(1 \otimes \a_{H^{*}}^{-1}\b_{H^{*}}^{-2}(f_{1})) \c \b_{M}^{-1}\om_{M}(m)]_{<0>})
\\&=\big<(\v\otimes \mathrm{id})[(1 \otimes \a_{H^{*}}^{-1}\b_{H^{*}}^{-2}(f_{1})) \c \b_{M}^{-1}\om_{M}(m)]_{<-1>}
\\&\quad(1 \otimes E_{H^{*}}^{-1}A_{H^{*}}D_{H^{*}}\a_{H^{*}}^{2}\b_{H^{*}}^{-2}\om_{H^{*}}^{3}(f_{2})) ,\, E^{-1}AD\a ^{2}\om ^{3}\psi^{2} (h) \big>
\\&\quad \a_{M}\psi_{M}^{-1}([(1 \otimes \a_{H^{*}}^{-1}\b_{H^{*}}^{-2}(f_{1})) \c \b_{M}^{-1}\om_{M}(m)]_{<0>})
\\&=\big<(\v\otimes \mathrm{id})[(1\o B_{H^{*}}^{-1}C_{H^{*}}^{-1}\a_{H^{*}}^{-1}\b_{H^{*}}^{-3}\psi_{H^{*}}^{-1}(f_{1}))
(\a \b^{-1}\om\o \a_{H^{*}}\b_{H^{*}}^{-1}\om_{H^{*}})(m_{<-1>})] ,\, 
\\&\quad E^{-1} AD\a ^{2}\om ^{3}\psi^{2} (h) \big>
\\&\quad \a_{M}\psi_{M}^{-1}((1\o \a_{H^{*}}^{-1}\b_{H^{*}}^{-2}\om_{H^{*}}( f_{2})) \c \b_{M}^{-1}\om_{M}(m_{<0>})).
\\&=\big<(\v\otimes \mathrm{id})m_{<-1>},\, E^{-1}AD\a \om^{2}\psi (h_2) \big>
\\&\quad f_{1}\b \psi (h_{1})(1 \otimes \b_{H^*}^{-2}\om_{H^*}\psi_{H^*}^{-1}(f_{2})) \c \a_{M}\b_{M}^{-1}\om_{M}\psi_{M}^{-1}(m_{<0>}).
\end{align*}
Since $f$ was arbitrary, we have
\begin{align*}
[m\tl AD\a \om (h_2)]_{(0)}\o\alpha ^{-1}\b \psi \omega ^{-1}(h_{1})[m\tl AD\a \om (h_2)]_{(1)}=m_{(0)}\tl AD\a \om\psi  (h_{1})\o\b (m_{(1)})h_{2}.
\end{align*}
The Eqs. $(\ref{e4.1})$, $(\ref{e4.2})$ and $(\ref{e4.4})$ left to the reader.

Therefore $M\in \mathcal{LR}(H)\binom{m,n,p,q}{s,t,u,v}$. It is straightforward to verify that any morphism in $\!^{H\o H^{*}}_{H\o H^{*}}\mathcal{YD}\binom{s,t}{p,q}$ is also a morphism in $\mathcal{LR}(H)\binom{m,n,p,q}{s,t,u,v}$. The proof is completed.
\end{proof}
To summarize, we have the main result of this section.
\begin{theorem}
Let $H$ be a BiHom-bialgebra. Then we have an isomorphism of prebraided monoidal category
\begin{align*}
\mathcal{LR}(H)\binom{m,n,p,q}{s,t,u,v}\cong \!^{H\o H^{*}}_{H\o H^{*}}\mathcal{YD}\binom{s,t}{p,q}.
\end{align*}
Furthermore, when $H$ is a BiHom-Hopf algebra with a bijective antipode, this isomorphism is braided.
\end{theorem}
\begin{proof}
Easy to see that functor $F$ is monoidal. Let $A=\a^{m}\psi^{n},B=\b^{p}\om^{q},C=\a^{s}\psi^{t},D=\b^{u}\om^{v},E=\a^{m+s+2}\b^{p+u+2}\om^{q+v+2}\psi^{n+t+2}$, for all $M,N\in \mathcal{LR}(H)\binom{m,n,p,q}{s,t,u,v}$ and $m\in M, n\in N$,
\begin{align*}
&(BC\b \psi \o B_{H^{*}}C_{H^{*}}\b_{H^*}\psi_{H^*})(m_{<-1>})\c\beta_{N}^{-1}\psi_N^{-1}\om_{N}(n)\otimes \om_M ^{-1}(m_{<0>})
\\&=(EBC\b \psi (m_{[-1]}) \otimes B_{H^{*}}C_{H^{*}}\b_{H^*}\psi_{H^*}(e^i)) \c\beta_{N}^{-1}\psi_N^{-1}\om_{N}(n)\otimes \om_M ^{-1}[
\a_M ^{-1}(m_{[0]})\tl E\a^{-1}\om^{-1} (e_{i})]
\\&=(EBC\b \psi (m_{[-1]}) \otimes e^i) \c\beta_{N}^{-1}\om_{N}\psi_N^{-1}(n)\otimes \a_M^{-1}\om_M^{-1}(m_{[0]})\tl AD\a \b \psi (e_i)
\\&=(EBC\b \psi (m_{[-1]}) \otimes \v)\c [(1 \otimes e^i) \c \beta_{N}^{-2}\om_{N}\psi_N^{-1}(n)]\otimes \a_M^{-1}\om_M^{-1}(m_{[0]})\tl AD\a \b ^{2}\psi (e_i)
\\&=BC\b \psi (m_{[-1]})\tr [(1 \otimes e^i) \c \beta_{N}^{-2}\om_{N}\psi_N^{-1}(n)]\otimes \a_M^{-1}\om_M^{-1}(m_{[0]})\tl AD\a \b ^{2}\psi (e_i)
\\&=BC\b \psi (m_{[-1]})\tr \b_N^{-1}\psi_N^{-1}(n_{(0)})\otimes \a_M^{-1}\om_M^{-1}(m_{[0]})\tl AD\a \om (n_{(1)}).
\end{align*} 
Hence, $F$ is prebraided.

The proof is completed.
\end{proof}
\section*{Acknowledgements}The author are deeply indebted to the referee for his/her very  useful suggestions and some improvements to this paper. This work was partially supported by the Scientific Research Foundation of Nanjing Institute of Technology (No. YKJ202219).

\end{document}